\documentclass[11pt]{report}
\usepackage{amsmath, amssymb, amsthm, amsfonts, amscd}
\usepackage[all]{xy}

\newcommand{\QQ}{{\mathbb{Q}}}
\newcommand{\QQbar}{{\overline{\mathbb{Q}}}}
\newcommand{\CC}{{\mathbb{C}}}
\newcommand{\HH}{{\mathbb{H}}}
\newcommand{\PP}{{\mathbb{P}}}

\newcommand{\RR}{{\mathbb{R}}}
\newcommand{\ZZ}{{\mathbb{Z}}}
\newcommand{\FF}{{\mathbb{F}}}
\newcommand{\TT}{{\mathbb{T}}}
\newcommand{\mm}{{\mathbf{m}}}
\newcommand{\dual}{\vee}
\newcommand{\goto}{\mapsto}
\newcommand{\compose}{\circ}
\newcommand{\st}{\Big\vert}
\newcommand{\dash}{-}
\newcommand{\old}{ {\mbox{old}}}
\newcommand{\new}{ {\mbox{new}}}

\newcommand{\isom}{\simeq}
\newcommand{\ndiv}{\nmid}
\newcommand{\tors}{\text{tors}}
\newcommand{\intersect}{\bigcap}
\DeclareMathOperator{\num}{Num}
\DeclareMathOperator{\alb}{Alb}
\DeclareMathOperator{\jac}{Jac}
\DeclareMathOperator{\End}{End}

\let\hom=\relax
\DeclareMathOperator{\hom}{Hom}
\DeclareMathOperator{\frob}{Frob}
\DeclareMathOperator{\ext}{Ext}
\DeclareMathOperator{\id}{Id}
\DeclareMathOperator{\trace}{trace}
\DeclareMathOperator{\Pic}{Pic}
\DeclareMathOperator{\op}{op}
\DeclareMathOperator{\ord}{ord}

\DeclareMathOperator{\aut}{Aut}

\DeclareMathOperator{\ann}{ann}
\DeclareMathOperator{\SL}{SL}
\DeclareMathOperator{\GL}{GL}
\DeclareMathOperator{\Gal}{Gal}

\newtheorem{thm}{Theorem}[section]
\newtheorem{lemma}[thm]{Lemma}
\newtheorem{prop}[thm]{Proposition}
\newtheorem{cor}[thm]{Corollary}
\theoremstyle{definition} \newtheorem{defn}[thm]{Definition}
\theoremstyle{remark} \newtheorem{remark}[thm]{Remark}
\theoremstyle{remark} 
\theoremstyle{remark} \newtheorem{conjecture}{Conjecture}

\begin{document}

\title{Modular Abelian Variety of Odd Modular Degree}
\author{Soroosh Yazdani}

\maketitle

\begin{abstract}
We will study modular Abelian varieties with odd congruence numbers, by
studying the cuspidal subgroup of $J_0(N)$. We show the conductor of such
Abelian varieties must be of a special type, for example if $N$ is odd
then $N=p^\alpha$ or $N=pq$ for some prime $p$ and $q$.
We then focus our attention to modular elliptic curves, and using result
of Agashe, Ribet, and Stein \cite{ARS}, we try to classify all elliptic
curves of odd modular degree. Our studies prove many cases of the Stein and
Watkins's conjecture on elliptic curves with odd modular degree.

\end{abstract}



\tableofcontents


\chapter{Preface}
\label{preface}
    After the work of Wiles, Taylor-Wiles, et al, we now know that all elliptic
curves over $\QQ$ are modular (see \cite{BCDT}), which implies that there is a map
$\pi:X_0(N) \rightarrow E$ defined over the rationals. As such, we have 
a new invariant attached to a given elliptic curve, namely the degree of
the modular uniformization $\pi$. This invariant is related to many other
invariants of an elliptic curve, for instance this number is closely related
to congruences between modular forms \cite{ARS}. Also we know that finding a good
bound on degree of $\pi$ in terms of $N$ is equivalent to the $ABC$ conjecture
\cite{MM}, \cite{Frey}.

After intense computer calculation, Watkins conjectured that 
$2$ to the power of the rank of $E(\QQ)$ divides the modular degree of
$E$. In particular, when the modular degree of $E$ is odd, then $E(\QQ)$ must
be finite.
Furthermore, Stein and Watkins have observed that such elliptic curves must 
have good reductions
away from at most two primes. The goal of this paper is to study
Stein and Watkins conjecture, and some generalization of their
conjecture to modular Abelian varieties.
As such we recall some definitions
and basic results with regards to modular Abelian varieties, congruences
between them, cuspidal subgroup, and elliptic
curves in chapter \ref{prelim}. In chapter \ref{chap2} we study 
general modular Abelian varieties with odd congruence numbers, and prove
many conditions they need to satisfy. Chapter \ref{chap3} is dedicated
to studying elliptic curves with odd modular degree, where we use results of
chapter \ref{chap2} and some diophantine equations to prove parts of
the conjecture of Stein and Watkins.

\chapter{Preliminaries}    
\label{prelim}
	\section{Jacobian Variety}
\label{sec1.1}
 In this section we study the natural inclusion of a curve into its Jacobian.
 For the purposes of this paper, we will present the Jacobian variety as a special
 case of the Albanese variety:
 \begin{defn}
	 For a given variety $V/k$, an Albanese variety $(\alb(V)/k,i)$ is a couple
	 consisting of an Abelian variety $\alb(V)/k$ and a rational map 
	 $i :V \rightarrow \alb(V)$ such that
	 \begin{enumerate}
		 \item image of $V$ under $i$ generates $\alb(V)$,
		 \item For every rational map $\pi:V \rightarrow B$ of $V$ into 
			 an Abelian variety $B$, there exists a homomorphism 
             $\pi_*:\alb(V) \rightarrow B$ and a constant $c(\pi) \in B(\overline{k})$
			 such that $\pi=\pi_*\circ i+c(\pi).$ 
	 \end{enumerate}
 \end{defn}
 To see the construction of this variety, and some of its properties, we refer the
 reader to \cite{langabvar}. 

 Note that if $p \in V(k)$ is a rational point of $V$, we can find an Albanese variety
 $(\alb(V),i)$ such that $i(p)=0$. We call this the canonical Albanese variety of
 $(V,p)$. 
 We denote $i$ by $i_V$ or $i_{V,p}$ whenever we want to emphasize $V$ and $p$.
 If $C/k$ is a curve and $p \in C(k)$ and $(\alb(C),i)$ is the canonical Albanese
 variety of $(C,p)$, then $\alb(C)$ is just the Jacobian of $C$ and we denote it
 by $\jac(C)$. We call the map $i_{C,p}$ the Albanese embedding of $C$ \cite{langabvar}.
 For a curve $C/k$,
 the Albanese map is easy to desribe. Specifically $i_{C,p}(z)=(z)-(p),$ and
 \begin{equation}
	 i_{C,q}(z)=i_{C,p}(z)+(p)-(q)=i_p(z)-i_p(q).
 \label{eq:2}
 \end{equation}

 Let $f:V \rightarrow W$ and let $(\alb(W),i_W)$ and $(\alb(V),i_V)$ be two 
 Albanese varieties of $V$ and $W$. Then we get a map 
 \[i_W \circ f:V \rightarrow \alb(W),\]
 which by definition of Albanese variety implies that we can find a map
 $(i_w \circ f)_*:\alb(V) \rightarrow \alb(W)$ and a constant $c(i_w\circ f)$
 such that $i_w\circ f=(i_w\circ f)_*\circ i_v + c(i_w \circ f).$
 We denote $(i_w\circ f)_*$ and $c(i_w \circ f)$ by $f_*$ and $c(f)$ respectively.
 Now let $p \in V(k)$ and let $(\alb(V),i_V)$ be the canonical Albanese variety of $(V,p)$.
 Furthermore, let $(\alb(W),i_W)$ be the canonical Albanese variety of $(W,f(p))$.
 Then we have the following commutative diagram:
 \begin{eqnarray}
 \xymatrix{
	 V \ar[r]^-{i_V} \ar[d]_f & \alb(V) \ar[d]^{f_*} \\
	 W \ar[r]^-{i_W} &  \alb(W).
 }
 \label{eq:1}
 \end{eqnarray}
 When $V=W=C$ we get the following
 \begin{prop}
	 Let $f:C \rightarrow C$ be a map of curves and $p\in C$, and let 
     $i=i_{C,p}:C \rightarrow \jac(C)$ be the canonical Albanese embedding.
	 Then there exists a map of varieties 
     $\alb(f): \jac(C) \rightarrow \jac(C)$
	 such that the following diagram
	 \[ 
	 \xymatrix{
		 C \ar[d]_f \ar[r]^-{i} & \jac(C) \ar[d]^{\alb(f)} \\
		 C \ar[r]^-{i} & \jac(C)
	 } \]
	 commutes. 
     Furthermore, if $A$ is an Abelian variety and 
     $\pi:\jac(C) \rightarrow A$ is a surjective map of Abelian
     varieties such that $f_*(\ker(\pi)) \subset \ker(\pi))$,
     then we can find a map (denoted by $f_*$ as well) 
     $A \rightarrow A$ such that
     the following diagram
	 \[ \xymatrix{
		 \jac(C) \ar[r]^-{\pi} \ar[d]_{f_*} & A \ar@{.>}[d] \\
		 \jac(C) \ar[r]^-{\pi} & A
	 } \]
     commutes.
     In this case, we can also construct a map $\alb(f):A \rightarrow A$ 
     of varieties such that the following diagram
	 \[ 
	 \xymatrix{
		 C \ar[d]_f \ar[r]^-{\pi \circ i} & A \ar@{.>}[d] \\
		 C \ar[r]^-{\pi \circ i} & A
	 } \]
	 commutes.
	 \label{albfunctor}
 \end{prop}
 \begin{proof}
     Note that by diagram \ref{eq:1} we can find
     $f_*$ such that the following diagram
     \begin{eqnarray}
     \xymatrix{
         C \ar[r]^-{i_{C,p}} \ar[d]_f & \jac(C) \ar[d]^{f_*} \\
         C \ar[r]^-{i_{C,f(p)}} &  \jac(W)
     }
     \end{eqnarray}
     commutes. However, $i_{C,p}(z)=i_{C,f(p)}(z)-i_{C,f(p)}(p)$
     which means that 
     \[ \alb(f)(z)=f_*(z)+i_{C,f(p)}(p)=f_*(z)+c(f) \] is
     the desired map. 
     To prove the second statement, for any $\alpha \in A$
     let $z \in \jac(C)$ such that $\pi(z)=\alpha$,
     and let $f_*(\alpha)=\pi(f_*(z)).$ This is independant
     of choice of $z$ by our assumption that $\ker(\pi)$
     is invariant under $f_*$.
     Finally, the last statement follows by 
	 noting that $f_*+\pi(c(f)):A \rightarrow A$ will make the 
	 above diagram commute.
 \end{proof}

 Let $G$ be a finite group that acts on the curve $C$. Then we can induce
 an action on $\jac(C)$ in two different way: For any $g\in G$ we have 
 a map $g_* : \jac(C) \rightarrow \jac(C)$ and 
 $\alb(g):\jac(C) \rightarrow \jac(C).$ We call the first one the 
 covariant action and the second one the Albanese induced action.
 Let $\pi: \jac(C) \rightarrow A$ be an Abelian variety quotient of $\jac(C)$ 
 such that $\ker(\pi)$ is invariant under $G$.
 Then, by \ref{albfunctor}, for all $g\in G$ we can construct 
 $\alb(g):A \rightarrow A$
 such that the following
 \begin{equation*}
	 \xymatrix{
	 C \ar[r] \ar[d]_g & A \ar[d]^{\alb(g)} \\
	 C \ar[r] &A
	 }
 \end{equation*}
 commutes.
 We have the following useful proposition:
 \begin{prop}
 	Let $G$ act on $C$. Let $\pi:\jac(C) \rightarrow A$ be a quotient of $\jac(C)$.
	Assume that the Albanese induced action exists and is trivial on $A$. 
	Then
	$\pi \circ i$ factors through $C \rightarrow C/G \rightarrow \jac(C/G).$
	\label{albquotient}
 \end{prop}
 \begin{proof}
 	Since the Albanese induced action is trivial on $A$, we have for any $g \in G$
    the following diagram commutes.
	\[ \xymatrix{
		C \ar[rr]^g \ar[rd] & & C \ar[ld] \\
		  & A &
	} \]
	Therefore, the map $\pi \circ i$ factors through $C \rightarrow C/G$. By the Albanese property
	of the Jacobian variety, we get that $\pi \circ i$ factors through $C \rightarrow C/G \rightarrow \jac(C/G)$,
	which is the desired result.
 \end{proof}

 Given an Abelian variety $A$, we can construct a dual Abelian variety $A^\dual=\Pic^0(A)$.
 We know that $A^\dual$ is isogeneous to $A$. We also know that $\End(A) \isom \End(A^\dual)^{\op},$
 via the map $\phi \in \End(A)$ going to $\phi^\dual$. If $\End(A)$ is a 
 commutative ring then $\End(A)^\dual \isom \End(A)$.
 Let $I$ be an ideal in $\End(A)$ and let
 \[ A[I] = \bigcap_{\phi \in I} \ker(\phi).\]
 The following proposition is well known.
 \begin{prop}
     Let $A$ be a simple Abelian variety such that $R=\End(A)$ is a commutative
     ring.
     If $I$ is not the unit ideal then $A[I]\neq 0$.
 \end{prop}
 \begin{proof}
     Since $R$ is finite over $\ZZ$ and has no zero divisors, $R$ is just an
     order in a number field. 
     Assume that $A[I]=0$. If $I$ is the zero ideal, then $A[I]=A$, so assume
     that $I$ is a non-zero ideal.
     Let $p \in I$ be a prime integer in $I$. Then we have $I/pR \subset R/pR$.
     Since $R/pR$ is an Artinian ring, we get that for some integer $n$
     and $J=(I/pR)^n$ we have that $J^2=J$. This implies that if $I$
     is not the unit ideal, then $I^n \subset pR$. Therefore 
     $A[p] \subset A[I^n]$, which implies that $A[I^n] \neq 0$. Let 
     $0 \neq P \in A[I^n]$. Then for any $\phi,\psi \in I$ we have that
     $\psi(\phi^{n-1}(P))=0$. Let $m$ be such that $\psi(\phi^{m}(P))=0$ for all 
     $\psi, \phi \in I$, but for some $\phi' \in I$ we have $\psi(\phi^{m-1}(P)) \neq 0$ 
     for any $\psi' \in I$. Let $Q=\phi'^{m}(P)$. We know that $Q \neq 0$ since we can 
     let $\psi'=\phi'$. On the other hand by our first assumption $\psi(Q)=0$ for any 
     $\psi \in I$. Therefore $0 \neq Q \in A[I]$.
 \end{proof}
 \begin{cor}
     If $A$ is a simple Abelian variety such that $R=\End(A)$ is commutative
     ring and $A[I] \neq 0$ then $A^\dual[I] \neq 0$.
 \end{cor}
 \begin{proof}
     If $A[I] \neq 0$, then $I$ is not the unit ideal (since $\ker(\id)=0$,)
     which implies that $I$ as an ideal of $\End(A^\dual)$ is not a unit
     ideal. Using proposition above we get that $A^\dual[I] \neq 0.$
 \end{proof}

\section{Modular Curves}
Let $X_0(N)$ be the moduli space of pairs $(E,C_N)$, where
$E$ is a generalized elliptic curve and $C_N$ is a cyclic subgroup of
order $N$. It turns out that $X_0(N)$ is in fact a curve, and there is
a smooth model of $X_0(N)$ over $\ZZ[1/N]$. Furthermore one has 
the complex curve $\overline{\HH}/\Gamma_0(N)=X_0(N)(\CC)$,
where $\HH$ is the complex upper half plane, $\overline{\HH}=\HH \cup \PP^1(\QQ)$, and 
\[
\Gamma_0(N) = \left\{ \begin{pmatrix} a & b \\ c & d \end{pmatrix} \st ad-bc=\pm 1, N|c \right\} \subset \SL_2(\ZZ).
\]
Here $\begin{pmatrix} a & b \\ c & d \end{pmatrix} z = {az+b \over cz +d }.$
The points of $X_0(N)(\CC)$ that are in the image of $\PP^1(\QQ)$ are called
the cusps of $X_0(N)$, so the cusps are in correspondence to $\PP^1(\QQ)/\Gamma_0(N)$.
(See \cite{DiamondIm} for the proof of above claims.)

Given $r | N$ such that $(N/r,r)=1$, we can decompose $C_N=C_r \times C_{N/r}$.
Therefore, we can find natural map 
$\alpha_{N,r}:X_0(N) \rightarrow X_0(r)$, where we just forget
about $C_{N/r}$, that is $(E,C_N) \goto (E,C_r)$. We also define
the degeneracy map $\beta_{N,r}:X_0(N) \rightarrow X_0(r)$ where
$(E,C_N) \goto (E/C_r, E[r]/C_r).$
We usually drop the subscript $N$ and $r$ from the notation, that
is we denote $\alpha, \beta$, 
to mean $\alpha_{N,r}, \beta_{N,r}$ respectively.
\begin{remark}
    If $N$, $a$, and $b$ are pairwise relatively prime, then we get
    the following diagram
    \[ \xymatrix{
        X_0(abN) \ar[r]^f \ar[d]^g &  X_0(aN) \ar[d]^g \\
        X_0(bN) \ar[r]^f & X_0(N)
    } \]
    commutes, where $f$ and $g$ are either $\alpha$ or $\beta$. 
    As such we get that
    these degeneracy maps commute with each other, whenever we are
    dealing with numbers relatively prime to each other.
    \label{DegeneracyCommute}
\end{remark}
When $N$ is square free
we denote the cusps of $X_0(N)$ by
$P_r \in X_0(N)$, indexed by $r | N$, in such a way that $P_r$ is
unramified under the degeneracy map $\alpha_{N,r}:X_0(N) \rightarrow X_0(r)$. 
Recall that $X_0(N)(\CC)$ is isomorphic to $\overline{\HH}/\Gamma_0(N)$.
Under this isomorphism, the cusp $P_r$ corresponds to the rational
number $r/N$. (So $P_1={1\over N} \equiv i\infty$.)

We let $J_0(N)=\jac(X_0(N))$ to be the Jacobian of the modular curve
$X_0(N)$. The maps $\alpha$ and $\beta$ induce maps on the respective
Jacobians. Specifically
\begin{eqnarray*}
 (\alpha_{N,r})_*, (\beta_{N,r})_* : J_0(N) \rightarrow J_0(r), \\
 (\alpha_{N,r})^*, (\beta_{N,r})^* : J_0(r) \rightarrow J_0(N).
\end{eqnarray*}
For $M$ an integer and $p$ a prime such that $p \ndiv M$ we define
$J_0(pM)_{p\mbox{-old}}$ to be the image of $J_0(M)$ in $J_0(pM)$ under
$(\alpha_{pM,M})^*$ and $(\beta_{pM,M})^*$, that is
\[ J_0(pM)_{p\mbox{-old}}=(\alpha_{pM,M})^*(J_0(M))+(\beta_{pM,M}))_*(J_0(M)).\]
Furthermore define $J_0(Mr)_{r\mbox{-old}}=\sum_{p | r} J_0(Mr)_{p\mbox{-old}}.$
Similarly we define 
\[ J_0(pM)_{p\mbox{-new}}=\ker( (\alpha_{pM,M})_*)+\ker( (\beta_{pM,M})_*),\]
and $J_0(Mr)_{r\mbox{-new}}=\bigcap_{p | r} J_0(Mr)_{p\mbox{-new}}.$
Also we define the $J_0(Mr)^{r\mbox{-old}}=J_0(Mr)/J_0(Mr)_{r\mbox{-new}}$ and
$J_0(Mr)^{r\mbox{-new}}=J_0(Mr)/J_0(Mr)_{r\mbox{-old}}.$ Finally let $J_0(N)_{\new}=J_0(N)_{N\mbox{-new}},$
and similarly for $J_0(N)^{\new}$, $J_0(N)_{\old}$, and $J_0(N)^{\old}.$

\subsection{Cuspidal Subgroup}
Let $N$ be a square free integer. We have the following
\begin{defn}
    The {\em cuspidal subgroup} of $J_0(N)$ is the subgroup $C \subset J_0(N)$
    generated by elements $P_r-P_1$ for $r |N$.
    \label{CuspidalDefn}
\end{defn}
In this subsection we study the order of elements in this group, and
calculate this order for a certain elements in the cuspidal subgroup.
These elements will later be used in finding congruences between
modular Abelian varieties (see section \ref{sec23}).

Recall the Dedekind's eta function is defined as
$$\eta(\tau)=q^{1/24}\prod_{n=1}^\infty (1-q^n)$$ where $q=e^{2\pi i \tau}.$
We also denote $\eta(M\tau)$ by $\eta_M(\tau)$.
Note that the $\eta$ has a zero of order $1/24$ at the cusps of 
$\overline{\HH}$, and away from the cusps
it is holomorphic and nonvanishing. 
We use $\eta_M$ to construct 
functions with divisors supported on the cusps. 
In particular $\eta_M$ has a zero of order 
\begin{equation}
    {1\over 24} {N d'^2 \over dt M},
    \label{CuspIsom1}
\end{equation}
at the cusp corresponding to $x/d \in \HH$,
where $d'=\gcd(d,M)$ and $t=\gcd(d,N/d)$ (see for example \cite{Ogg74}). 
Let ${\mathbf r}=(r_\delta)$ be a family of rational numbers $r_\delta \in \QQ$
indexed by the positive divisors of $\delta | N$.  Then the divisor of
function $g_{\mathbf r}=\prod_{\delta | N} \eta_\delta^{r_\delta}$ is 
supported on the cusps and we can calculate this divisor explicitly.
Note that
$R=\left\{ g_{\mathbf r} \st r_\delta \in \QQ \right\}$ forms a vector space
of dimension $2^t$ under multiplication, with basis $\eta_\delta$. The
discussion above gives us an isomorphism between $R$ and the 
rational vector space generated by the cusps of $X_0(N)$, call it $S$. 
We will give a more managable description of this isomorphism when 
$N=p_1p_2\dots p_t$ is square free.
First we
will define an explicit isomorphism between $R$ and
$V_1 \otimes V_2 \otimes \cdots \otimes V_t$ where each $V_i$ is $2$
dimensional with a chosen basis $e_{i,0}$ and $e_{i,1}$. 
Specifically $e_{1,k_1} \otimes e_{2,k_2} \otimes \cdots \otimes e_{t,k_t}$
is mapped to $\eta_\delta$ where $\delta=p_1^{k_1}\dots p_t^{k_t}$. 
Similarly we have that $S$ is isomorphic to 
$W_1 \otimes W_2 \otimes \cdots \otimes W_t$ where $W_i$ is generated by
$f_{i,0}$ and $f_{i,1}$ and $\otimes f_{i,k_i}$ is mapped to
$P_{p_1^{k_1}\dots p_t^{k^t}}.$
Now define $\Lambda_k : V_k \rightarrow W_k$ to
be $\Lambda_k(e_{k,\sigma})=p_k^\sigma f_{k,1}+p_k^{1-\sigma} f_{k,0},$
that is $\Lambda_k=\begin{pmatrix} p_k & 1 \\ 1 & p_k \end{pmatrix}.$
Let
\begin{equation}
    \begin{array}{rcl}
        \Lambda: V_1 \otimes \cdots \otimes V_t &\rightarrow& W_1 \otimes \cdots \otimes W_t, \\
        v_1 \otimes \cdots \otimes v_t & \goto & {1 \over 24} \Lambda_1(v_1) \otimes \cdots \otimes \Lambda_t(v_t).
    \end{array}
    \label{CuspIsom2}
\end{equation}
Note that $\Lambda^{-1}=24 \bigotimes_{i=1}^t \Lambda_i^{-1}$.

Our main tool is the following
\begin{prop}
    Let ${\bf r}=(r_\delta)$ be a family of rational number $r_\delta \in \QQ$
    indexed by all the positive divisors of $\delta | N$. Then the function
    $g_{\bf r}=\prod_{\delta | N} \eta_\delta^{r_\delta}$ is a modular function
    on $X_0(N)$ if and only if the following conditions are satisfied:
    \begin{enumerate}
        \item all the $r_\delta$ are rational integers;
        \item $\sum_{\delta | N} r_\delta \delta \equiv 0 \pmod {24}$;
        \item $\sum_{\delta | N} r_\delta{N \over \delta} \equiv 0 \pmod {24}$ ;
        \item $\sum_{\delta | N} r_\delta=0$; 
        \item $\prod_{\delta | N} \delta^{r_\delta}$ is a square of a rational number.
    \end{enumerate}
    \label{ligozatprop}
\end{prop}
For the proof see \cite{Ligozat75}. This proposition along with the 
isomorphism \ref{CuspIsom2} gives us a recipe for calculating the order
of specific elements in the cuspidal subgroup, and the group structure of the cuspidal subgroup.
In \cite{KiatLing} this was done in the case $N$ is the product of two primes.
When $N$ is square free, the proposition \ref{ligozatprop} simplifies to
\begin{lemma}
    Let $N=p_1\dots p_t$ be a square free number. An integral element 
    $w \in W$ is linearly equivalent to the zero cusp if and only if
    \begin{enumerate}
        \item $\Lambda^{-1}w$ is integral,
        \item $\left( (1,1)\otimes (1,1) \otimes \cdots \otimes (1,1)\right) w=0$ (that is
            $w$ is a degree $0$ divisor).
        \item for all $i$ we have 
            \[ (1,1)\otimes \cdots \otimes (0,1) \otimes \cdots \otimes (1,1) \Lambda^{-1}w,\]
            is even (where the $(0,1)$ vector is in the $i$-th position, and every
            other vector is the $(1,1)$ vector).
    \end{enumerate}
    \label{cusplemma2}
\end{lemma}
\begin{proof}
    For an integral $w \in W$ to be linearly equivalent to $0$, 
    $v=\Lambda^{-1}w$ must satisfy the conditions in proposition
    \ref{ligozatprop}. 
    \begin{enumerate}
        \item We are specifically asking $\Lambda^{-1}w$ to be integral, so 
            condition 1 is satisfied by assumption.
        \item Note that the sum $\sum_{\delta | N} \delta r_\delta$ 
            is the same as 
            \[\left( (1,p_1) \otimes (1,p_2) \otimes \dots \otimes (1,p_t) \right) v.\]
            Substituting $v=\Lambda^{-1}w$ we get
            \begin{eqnarray*}
                \sum_{\delta | N} \delta r_\delta &=&  \left( \bigotimes_{i=1}^t (1,p_i) \right) v \\
                &=& \left( \bigotimes_{i=1}^t (1,p_i) \right)\Lambda^{-1}w \\
                &=& {24 \over \prod_{i=1}^t (p_i^2-1)}\left( \bigotimes_{i=1}^t (1,p_i) \right)\left( \bigotimes_{i=1}^t 
                 \begin{pmatrix} 
                     p_i & -1 \\ -1 & p_i 
                 \end{pmatrix} \right) w \\
                &=& {24 \over \prod_{i=1}^t (p_i^2-1)} \left( \bigotimes_{i=1}^t (0,p_i^2-1) \right) w \\
                &=& 24 \left( \bigotimes_{i=1}^t (0,1) \right)w
            \end{eqnarray*}
            Since we are assuming that $w$ is integral, we get that the above
            sum is divisible by $24$, so condition two is automatically satisfied.
        \item Similarly, note that the sum $\sum_{\delta | N} {N \over \delta} r_\delta$
            is the same as
            \begin{eqnarray*}
                \left( (p_1,1) \otimes (p_2,1) \otimes \dots \otimes (p_t,1) \right) v 
                = 24 \left( \bigotimes_{i=1}^t (1,0) \right)w.
            \end{eqnarray*}
            Again, the above sum is divisible by $24$, since we are assuming that 
            $w$ is an integral vector.
        \item Similarly $\sum_{\delta | N} r_\delta$ can be calculated by 
            \[\left( (1,1) \otimes (1,1) \otimes \dots \otimes (1,1) \right) v.\]
            Expanding this we get the product
            \[K\left( (1,1) \otimes (1,1) \otimes \dots \otimes (1,1) \right) w,\]
            for some nonzero $K$. Therefore we get that $\sum r_\delta=0$
            if and only if
            \[\left( (1,1) \otimes (1,1) \otimes \dots \otimes (1,1) \right) w=0,\]
            as desired.
        \item Finally $\prod_{\delta | N} \delta^{r_\delta}$ is a perfect
            square if the power of each prime is even. The power of the $i$-th
            prime of that product is just
            \[ (1,1)\otimes \cdots \otimes (0,1) \otimes \cdots \otimes (1,1) v,\]
            which gives us the desired result.
    \end{enumerate}
\end{proof}
As an immediate application of the lemma \ref{cusplemma2} we will calculate the order of cusps of the form 
$w_1 \otimes w_2 \otimes \cdots \otimes w_t \in W \isom S$
where $w_i=f_{0,i}\pm f_{1,i}$ for all $i$. This is a well known generalization of
the work of Ogg (\cite{Ogg74}):
\begin{prop}
    Let $N=p_1\dots p_t$ be a square free integer.
    Let $b_k = \pm 1$ for $k=1,2,\dots,t$, such that $b_k=-1$ for at least 
    one of these $k$'s, and \[ z=\sum_{d | N} \left(\prod_{p_k | d} b_k\right) P_d. \]
    Then, if $N=p_1$ is a prime, $z$ has order $\num\left( {p_1-1 \over 12} \right)$,
    otherwise it has order
    \[ \num\left( {(p_1+b_1)\cdots (p_t+b_t) \over 24} \right).\]
    \label{GeneralizedOgg}
\end{prop}
\begin{proof}
    Since the cusp $z$ has degree $0$, we only need to check for what value
    of $n$ does $nz$ satisfy conditions one and three of lemma \ref{cusplemma2}.
    The cusp $z$ maps to $\bigotimes w_i$ where $w_i=f_{0,i} + b_i f_{1,i}.$
    Therefore 
    \[\Lambda^{-1}w={24 \over {\prod (p_i^2-1)}} \bigotimes_{i=1}^t \left((b_i p_i-1)f_{0,i}+(p_i-b_i)f_{1,i}\right).\]
    Factoring $p_i-b_i$ we get 
    \[ \Lambda^{-1}w={24 \over \prod (p_i+b_i)} \bigotimes_{i=1}^t (b_if_{0,i}+f_{1,i}).\]
    Therefore $\num\left({\prod (p_i+b_i) \over 24}\right)$ divides $n$.
    As for condition three, note that $(1,1) (b_if_{0,i}+f_{1,i})=b_i+1$
    which is even, so as long as $t>1$ condition three is automatically
    satisfied, and hence the order is $\num\left({\prod (p_i+b_i) \over 24}\right)$. 
    If $t=1$ then we get that that the order is $\num\left({p_1-1 \over 12}\right).$
\end{proof}

\subsection{Hecke Operators}
For an integer $N$ and a prime $p \ndiv N$, we 
have two degeneracy maps
\[ \alpha_{pN,N},\beta_{pN,N}:X_0(Np) \rightarrow X_0(N). \]
These two define a correspondence which is called the 
{\em $p$-th Hecke correspondence} on $X_0(N)$. 
This correspondence induces the Hecke operator $T_p$ 
on $J_0(N)$ via
\[ \xymatrix{
    T_p : J_0(N) \ar[r]^{\alpha^*} & J_0(Np) \ar[r]^{\beta_*} & J_0(N)
}. \]
Note that \[T_p( (E,C) ) = \sum_{D} ( (E/D, (C+D)/D),\]
where $D$ runs through all the cyclic subgroups of order $p$.
For general $n$, define $T_n$ by
\[T_n( (E,C) ) = \sum_{D} ( (E/D, (C+D)/D),\]
where $D$ runs through cyclic subgroups of order $n$ such that $C \cap D=0$.
We have the following well known identities
\begin{align*}
    T_{p^{k+1}}&= T_{p^k}T_p-pT_{p^{k-1}} &\mbox{if $p \ndiv N$,} \\
    T_{l^{k}} &=T_p^k & \mbox{if $l | N$,} \\
    T_{mn} &=T_m T_n & \mbox{if $(m,n)=1$.} 
\end{align*}

For $r|N$ such that $(r,N/r)=1$, we define the {\em Atkin-Lehner operator 
at $l$}, denoted by $w_r$, acting on $X_0(N)$ as follows:
Let $(E,C_N) \in X_0(N)$. By our assumption on $r$ we have
that $C_N = C_{r} \times C_{N/r}.$ Then
\begin{eqnarray*}
    w_r :  X_0(N) & \rightarrow  & X_0(N) \\
     (E,C_r \times C_{N/r}) & \goto & (E/C_r, (E[r]/C_r)\times (C_{N/r}+C_r)/C_r).
\end{eqnarray*}
Note that $w_r(P_1)=P_r$.
Also, note that $\beta_{N,N/r}=\alpha_{N,N/r} \compose w_r$, just by 
unraveling the definitions.

The set of Hecke operators form a subalgebra
of $\End(J_0(N))$. We denote this algebra by $\TT=\ZZ[T_2,T_3,\dots],$
and call it {\em the Hecke algebra of level $N$.}
For $T \in \TT$ we have that $T(J_0(N)_{\new}) \subset J_0(N)_{\new}$.
Therefore, we can consider the image of $\TT$ in $\End(J_0(N)_{\new}),$
call this $\TT^{\new}$.
Note that even though Hecke operators commute with each other, in general
Hecke operators do not commute with the Atkin-Lehner operators. For example,
one can check that
$w_N T w_N$ is not necessarily $T$, rather $w_N T w_N=T^\dual$, the
action of $T$ induced on $J_0(N)^\dual$ (For details see \cite{DiamondIm}).
However, one can check that the Atkin-Lehner operators commute with the Hecke
operators over $J_0(N)_{\new}$. Therefore, in
$\TT^{\new}$ we have that $T_n^{\dual}=T_n$.

The action of the $p$-th Hecke operators is very easy to compute 
on $J_0(N)_{\FF_p}$ 
\begin{thm}[Eichler-Shimura Relation]
    On $J_0(N)_{\FF_p}$ we have that \[ T_p=\frob_p + p/\frob_p,\]
    for all $p \ndiv N$.
\end{thm}

\section{Modular Abelian Varieties}
If $I$ is a saturated ideal of $\TT$, then $A_I=J/IJ$ is an optimal
quotient of $J$ since $IJ$ is an Abelian subvariety. Let $\phi:J \rightarrow A_I$
be the quotient map. Then $(A_I)^\dual$, the dual of $A_I$, is the unique 
Abelian subvariety of $J$ such 
that it projects isogeneously to $A_I$. By the Hecke equivariance of $\phi$ it follows that
$A_I^\dual$ is $\TT$-stable, and hence $\TT$ acts on $A_I^\dual$.
Note that we also have an action of $\TT$ which comes from embedding 
$A_I^\dual \rightarrow J^\dual$. As we discussed in previous section, when
$A_I$ is a new Abelian variety, these two actions are the same. 
As result, we focus on $A_I$ only when $A_I$ is a new modular Abelian variety.

\subsection{Algebraic Congruences}
\begin{defn}
The {\em algebraic congruence group} is the quotient group
\[ {S_2(\Gamma_0(N),\ZZ) \over S_2(\Gamma_0(N),\ZZ)[I]+S_2(\Gamma_0(N),\ZZ)[I]^{\perp}}.
\]
If $A$ is an optimal quotient associated to $I$, we denote the above group
by $C_A$.
We call the order of the above group the {\em congruence number} of
$A$, and its exponent the {\em congruence exponenet} of $A$. Let $r_A$ denote the
congruence number of $A$ and $\widetilde{r_A}$ denote the congruence exponent of $A$.
\end{defn}
This group can be calculated from the Hecke algebra using the following
lemma:
\begin{lemma}
	Let $\phi:J_0(N) \rightarrow A$ be a new optimal quotient. Let
	$B = \ker(\phi)$.
	Let $\TT_1$ be
	the image of $\TT$ in $\End(A^\vee)$ and $\TT_2$ be the image
	of $\TT$ in $\End(B)$. Then
	\[ \hom\left({\TT_1 \oplus \TT_2 \over \TT},\QQ/\ZZ\right) \isom 
	{S_2(\Gamma_0(N),\ZZ) \over S_2(\Gamma_0(N),\ZZ)[I]+S_2(\Gamma_0(N),\ZZ)[I]^{\perp}}
	\]
    as Hecke modules. Specifially $(\TT_1 \oplus \TT_2)/\TT$ is the Pontryagin dual of
    $C_A$ and $(\TT_1 \oplus \TT_2)/\TT \isom C_A$ as finite abelian groups.
    (See also \cite{ARS}, lemma 4.3.)
	\label{HeckeCong}
\end{lemma}
\begin{proof}
    As in \cite{ARS}, apply the $\hom(\dash,\ZZ)$ functor to the
    \[ \xymatrix{
    0 \ar[r] & \TT \ar[r] & \TT_1 \oplus \TT_2 \ar[r] & (\TT_1 \oplus \TT_2)/ \TT \ar[r] & 0,
    } \]
    to get
    \[ \xymatrix{
    0 \ar[r] & \hom(\TT_1 \oplus \TT_2, \ZZ) \ar[r] & \hom(\TT,\ZZ) \ar[r] & \ext^1( (\TT_1 \oplus \TT_2)/\TT,\ZZ) \ar[r] & 0.
    } \]
    Note that the $0$ on the left is because $(\TT_1 \oplus \TT_2)/\TT$ is finite group, and 
    $\ZZ$ has no torsion subgroups, and the $0$ on the right is because $\TT_1 \oplus \TT_2$
    are torsion free, hence free $\ZZ$-modules.
    Using the $\TT$-equivariant perfect bilinear pairing $\TT \times S_2(\ZZ) \rightarrow \ZZ$
    given by $(t,g) \rightarrow a_1(t(g))$, the above exact sequence transforms to
    \[ 
    0 \to S_2(\Gamma_0(N),\ZZ)[I] \oplus S_2(\Gamma_0(N),\ZZ)[I]^\perp \to S_2(\Gamma_0(N)) \to \ext^1( (\TT_1 \oplus \TT_2)/\TT,\ZZ) \to 0.\]
    Therefore \[ C_A \isom \ext^1( (\TT_1 \oplus \TT_2)/\TT, \ZZ).\]
    Now for any torsion $\ZZ$-module $M$, applying $\hom(M,\dash)$ to the exact sequence
    \[ \xymatrix{
        0 \ar[r] & \ZZ \ar[r] & \QQ \ar[r] & \QQ/\ZZ \ar[r] & 0,
    } \]
    we get
    \[ \xymatrix{
        \hom(M,\QQ) \ar[r] & \hom(M,\QQ/\ZZ) \ar[r] & \ext^1(M,\ZZ) \ar[r] & \ext^1(M,\QQ).
    } \]
    However, $\hom(M,\QQ)=0$ since $M$ is assumed torsion, and $\ext^1(M,\QQ)=0$ since $\QQ$
    is divisible, and hence injective. 
    Therefore \[ C_A \isom \hom( (\TT_1 \oplus \TT_2)/\TT, \QQ/\ZZ).\]
    Since $(\TT_1\oplus \TT_2)/\TT$ is a torsion $\ZZ$-modules we have that
    $(\TT_1 \oplus \TT_2)/\TT \isom \hom( (\TT_1 \oplus \TT_2)/\TT, \QQ/\ZZ)$, which
    proves the last asserstion in the lemma. 
\end{proof}

\subsection{Geometric Congruences}
Let $\phi:J \rightarrow A$ be an optimal quotient. Dualizing this we get
$\phi^\dual: A^\dual \rightarrow J^\dual$. Composing this with the theta polatization we
get
\[ \psi:A^\dual \rightarrow J^\dual \isom J \rightarrow A.\]
\begin{defn}
    The {\em geometric congruence group} is the kernel of the isogeny $\psi.$ We denote the
    above group by $D_A$. The 
    {\em geometric congruence number} of $A$ is the order of $D_A$, and the {\em geometric
    congruence exponent} is the exponent of $D_A$.
    Let $n_A$ denote the geometric congruence number of $A$ and $\widetilde{n_A}$ denote the
    geometric congruence exponent of $A$.
\end{defn}
\begin{remark}
    Note that in \cite{ARS}, what we call {\em geoemtric congruence number} is called
    {\em modular number}, and {\em geometric congruence exponent} is called {\em modular
    exponent}. 
\end{remark}

\begin{remark}
    Note that
    \[ D_A=\ker(\psi) = \ker(\phi) \intersect A^\dual,\]
    since the map $\phi^\dual : A^\dual \rightarrow J^\dual$ is injective.
\end{remark}

For the rest of this section we briefly discuss the relationship between algebraic congruence
group and the geometric congruence group. In \cite{ARS} the following theorem is proved.
\begin{thm}[Agashe-Ribet-Stein]
    \label{ARSCongExp}
    If $f \in S_2(\CC)$ is a newform, then
    \begin{enumerate}
        \item We have $\widetilde{n_{A_f}} | \widetilde{r_{A_f}}$, and
        \item If $p^2 \ndiv N$, then $\ord_p(\widetilde{r_{A_f}}) = \ord_p(\widetilde{n_{A_f}}).$
    \end{enumerate}
\end{thm}

Here we prove a result along the line of first part of the above theorem.
\begin{lemma}
    Let $\phi:J_0(N) \rightarrow A$ be a new optimal quotient and let $B=\ker(\phi)$.
    Let $\mm$ be a maximal ideal of $\TT$. If $A^\dual[\mm]$ and $B[\mm]$ are both
    nontrivial, then $\mm$ is in the annihilator of $\hom(C_A,\QQ/\ZZ)$. Specifically
    $\#\TT/\mm$ divides $r_A$, and the characteristic of $\TT/\mm$ divides $\widetilde{r_A}$.
    \label{AlgCongLemm}
\end{lemma}
\begin{proof}
    Since $A^\dual[\mm]$ is nontrivial, $\TT_1 \otimes_\TT \mm$ is not the unit ideal,
    which implies $\TT_1 \otimes \TT/\mm$ is nontrivial.
    Similarly $B[\mm]$ nontrivial implies $\TT_2 \otimes \TT/\mm$ is nontrivial.
    Therefore $\TT/\mm \rightarrow (\TT_1 \oplus \TT_2) \otimes \TT/\mm$ is not surjective,
    since $(\TT_1 \oplus \TT_2) \otimes \TT/\mm$ has zero divisors, while $\TT/\mm$ is a field.
    Therefore 
    \[ \left({\TT_1 \oplus \TT_2 \over \TT} \right) \otimes_\TT \TT/\mm \isom 
    \hom(C_A,\QQ/\ZZ) \otimes \TT/\mm \]
    is nontrivial, which is the desired result.
\end{proof}

\begin{cor}
    If $A$ is a new optimal quotient then $\sqrt{\ann(D_A)} \subset \sqrt{\ann(C_A^\dual)},$
    where $\sqrt I$ is the product of prime ideals dividing $I$.
\end{cor}
\begin{proof}
    This follows immediately from \ref{AlgCongLemm}.
\end{proof}
\begin{cor}
    If $l$ is a prime number such that $l | n_A$ then $l | r_A$.
\end{cor}

\section{Elliptic Curves}
An Abelian variety of dimension $1$ is called an {\em elliptic curve}. Let
$R$ be any ring. An
elliptic curve over the ring $R$ has a model
\[ E:y^2+a_1xy+a_3y=x^3+a_2x^2+a_4x+a_6, \]
where $a_i \in R$, which is called a {\em Weierstrass model} 
(see \cite{AEC:I}). This model is not unique,
and for any $u, r,$ and $s$ we can find an equivalent model by
the following substitutions:
\begin{align*}
    ua_1'=& a_1+2s,\\
    u^2a_2'=& a_2-sa_1+3r-s^2, \\
    u^3a_3'=& a_3+ra_1+2t, \\
    u^4a_4'=& a_4-sa_3+2ra_2-(t+rs)a_1+3r^2-2st, \\
    u^6a_6'=& a_6+ra_4+r^2a_2+r^3-ta_3-t^2-rta_1.
\end{align*}
There are many invariants attached to an elliptic curve $E/R$. Here,
we recall the {\em discriminant} and the {\em conductor}.
\begin{prop}
    The discriminant $\Delta$ of the Weierstrass model
    \[ E: y^2+a_1xy+a_3y=x^3+a_2x^2+a_4+a_6, \]
    can be calculated as follows:
    \begin{align*}
        b_2&=a_1^2+4a_2, \\
        b_4&= 2a_4+a_1a_3, \\
        b_6&= a_3^2+4a_6, \\
        b_8&= (b_2b_6-b_4^2)/4, \\
        \Delta&= -b_2^2b_8-8b_4^3-27b_6^2+9b_2b_4b_6.
    \end{align*}
\end{prop}
Applying the above substitutions we get that the discriminant
of \[ E':y^2+a_1'xy+a_3'y=x^3+a_2'x^2+a_4'x+a_6'\]
is $\Delta'=u^{-12} \Delta$.
If $K$ is a local field with valuation $v$, and $E/K$ is an elliptic curve
over $K$, then one can choose a
Weierstrass model for $E$ such that $v(\Delta_E)$ is minimal.
If $K$ is a global field, and $E$ is defined over $K$, then we define
the minimal discriminant ideal 
\[ D_{E/K}= \prod_{v \in M_K} p_v^{v(\Delta_{E \times K_v})},\]
to be the product of minimal discriminants at each local place.
If $K$ is a global field with class number $1$, then one can choose a
Weierstrass model of $E$ such that $\Delta_E {\mathcal{O}}_K=D_{E/K}.$
If $K=\QQ$ then the number $\Delta_E$ is unique, and we can give an
interpretation for the sign of this number.
\begin{prop}
    Let $E/Q$ be an elliptic curve over the rationals. Then $E(\RR)$ has
    only one component if and only if $\Delta_E<0.$
\end{prop}

To define the conductor of an elliptic curve, we first recall that
given an elliptic curve $E/K$ and any integer $n$ coprime to the 
characteristic of $K$ we can construct a
Galois representation $\rho_{E,n}$ by studying the action of the Galois group on the
$n$-torsion points of $E$,
\[ \rho_{E,n} : \Gal(\overline{K}/K) \rightarrow \aut(E[n]) \isom \GL_2(\ZZ/n\ZZ).\]
Similarly, given a prime $l$, we can construct the Tate module of $E$ via
$T_lE=\lim_{\leftarrow} E[l^n]$. We can also study the action of the Galois
group on the Tate module
\[ \rho_{E,l^\infty} : \Gal(\overline{K}/K) \rightarrow \GL_2(\ZZ_l). \]
Let $K$ be a local field with residue characteristic $p$, and let $E$
be an elliptic curve over $K$. Let $l$ be a prime distinct from $p$.
Then we can calculate the Serre conductor of $\rho_{E,l^\infty}$. It turns
out that this conductor is independant of our choice of $l$
(see for example \cite{AECII}, or \cite{RibetStein}). 
We denote this number by $\delta(E/K)$ and we call it the conductor
of $E$.
When $\delta(E/K)=0$
we say that $E$ is unramified. In this case, $\rho_{E,l^\infty}$ is defined
by knowing the image of a $\frob_p \in \Gal(\overline{K}/K).$
Furthermore, when $E/K$ is unramified, then the minimal discriminant of
$E$ has valuation $0$.

Let $K$ be a global field and let $E$ be an elliptic curve over $K$. Then
we define the conductor of $E$ to be the product local conductors,
specifically
    \begin{eqnarray}
        \delta(E/K)= \prod_{v \in M_K^0} p_v^{\delta(E/K_v)}.
        \label{conductordefn}
    \end{eqnarray}
    Recall that the conductor of $E$ divides the discriminant of $E$ (see \cite{AECII}).
Therefore for almost all $v \in M_K^0$ we have
$E$ is unramified over $K_v$. When $E$ is unramified over $K_v$ we define
$a_v(E)$ to be
$\trace(\rho_{E,l^\infty}(\frob_v))$ for some choice of $\frob_v$.

Elliptic curves over complex numbers have a particularly easy description.
Specifically given an elliptic curve $E/\CC$ we can find a complex number 
$\tau \in \CC$ such that $E(\CC)\isom \CC/(\ZZ+\tau \ZZ)$, where the
isomorphism is in the category of complex curves. We denote $\CC/(\ZZ+\tau\ZZ)$
by $E_\tau$.

\subsection{Modular Uniformization}
Let $E/\QQ$ be an elliptic curve over the rationals. Then the conjecture of
Shimura and Taniyama as proved by Breuil, Conrad, Diamond, Taylor, Wiles, etc
says that $E$ is modular (\cite{W95}, \cite{TW95}, \cite{BCDT}).
This means that there is a normalized modular eigenform $f_E \in S_2(\Gamma_0(N))$ such 
that $T_p(f_E)=a_p(E) f_E$ for almost all primes $p$, where $a_p(E)$ is the trace of Frobenius of
$\rho_{E,l^\infty}$.
By Serre's epsilon conjecture \cite{Rib91}, the minimum number that $N$ can
be is the conductor of $E$.
Alternatively this means that there is a surjective map $\pi:X_0(N) \rightarrow E$.
We recall the construction of this map $\pi$.
For any modular eigenform $f \in S_2(\Gamma_0(N))$ with integer coefficients,
the construction in section 3 produces an Abelian variety of dimension $1$,
which is an elliptic curve. Therefore for any such $f$ we have a map
$J_0(N) \rightarrow E_f$ such that the kernel is an Abelian variety as well.
Embedding $X_0(N)$ in $J_0(N)$ we have a map $X_0(N) \rightarrow E_f$ which 
is surjective since $\pi_f:J_0(N) \rightarrow E_f$ was surjective. If we choose
$f=f_E$ we get that $E_f$ is isogeneous to $E$, and composing $X_0(N) \rightarrow E_f$ 
by this isogeny we get the desired $\pi$. 
When $\pi_f=\pi$ then we call $E$ the {\em optimal elliptic curve}.
Given an optimal elliptic curve $E$, we define $\deg(\pi_{f_E})$ to be the {\em modular
degree} of $E$. 
\begin{prop}
    Let $E/\QQ$ be an elliptic curve over rationals, and let $f$ be the
    modular form associated to $E$. Assume that $E=E_f$, that is $E$ is
    an optimal elliptic curve. Then the geometric congruence exponent of $E_f$
    is the same as the modular degree of $E$.
\end{prop}
\begin{proof}
    Recall that geometric congruene exponent is the exponent of the kernel of
    \[ E^\dual \rightarrow J_0(N) \rightarrow E.\]
    We know that $E^\dual=E$ since $E$ is an elliptic curve, and one can
    check that the above composition $E \rightarrow E$ is just multiplication
    by the modular degree. Therefore the exponent of the kernel of this map
    is exactly the modular degree.
\end{proof}

We can calculate the map $X_0(N) \rightarrow E$ explicitly over $\CC$. 
Recall that $X_0(N)(\CC) = \overline{\HH}/\Gamma_0(N)$ and $E(\CC)=\CC/(\ZZ+\tau\ZZ).$
Let $f_E$ be the modular form associated to $E$. Then the map
\begin{eqnarray*}
    \begin{array}{rcccl}
        \pi : & X_0(N)(\CC) &\rightarrow & E(\CC) =\CC/\Lambda_E \\
        & z & \goto & 2\pi i \int_z^\infty f_E(z)dz & \pmod {\Lambda_E}
    \end{array}
\end{eqnarray*}
where $\Lambda_E$ is generated by $2\pi i \int_z^\infty f_E(\gamma(z))dz$
for all $\gamma \in \Gamma_0(N).$ We let $\Lambda_E \cap \RR=\Omega_E \ZZ$, and
we call $\Omega_E$ the real period of $E$. To verify that this makes sense, one
we refer the reader to \cite{Shim94}.
	
\chapter{Modular Abelian Varieties with Odd Congruence Number}  
\label{chap2}
	In this chapter we will study simple modular Abelian varieties with odd congruence
numbers. By studying the twists of modular Abelian varieties, the action of the 
Atkin-Lehner involutions, the order of cuspidal subgroup, we show that if we have 
an absolutely simple modular abelian variety with odd congruence number, then
it has 
conductor $p^\alpha$, $pq$, or $2^{1+\alpha} N$ for some positive integer
$\alpha$ and odd prime $q$.

\section{Non-Semistable Case}
\label{CMcase}
The goal of this section is to prove the following
\begin{thm}
    \label{InnerTwist}
    Let $A$ be an absolutely simple modular Abelian variety $A$ of level $N$ with
    an odd congruence number. Let $\delta_p=0$ for odd primes $p$, and $\delta_2=2$.
    Assume that $p^{2+\delta_p} | N$. Then $A$ has good reduction 
    away from $p$ and $2$. Specifically if $p$ is odd, then $N=p^s$, $N=4p^s$, or
    $N=8p^s$ for $s \geq 2$, and if $p=2$ then $N=2^s$.
\end{thm}
We expect that something stronger is true. Specifically, the theorem should be
true without the absolutely simple assumption, however at this moment we don't know
how to overcome the difficulty with inner forms in that case. 
To prove this theorem we use the technique of Calegari and Emerton to
show that such modular Abelian varieties have inner twists
by a character of conductor $p$ \cite{CE}. 
Using results of Ribet on inner
twists \cite{RibetTwists}, we will prove that $A$ must have potentially good reduction
everywhere if $A$ is absolutely simple, and ultimately $A$ has good reduction away from
$p$. A key part of this argument is the following
\begin{lemma}
    If $\End_{\QQbar}(A)\otimes \QQ$ is a matrix algebra then
    $A$ is not absolutely simple.
    \label{TateLemma}
\end{lemma}
\begin{proof}
    Assume that $R=\End_{\QQbar}(A)\otimes \QQ$ is a matrix algebra. Then
    we can find projections $P_1, P_2 \in R$ such that $P_1+P_2=\id$,
    $P_1 P_2=0$ and $P_1, P_2 \not \in \{ 0, \id\}.$
    Now for some integer $n$ we have that 
    $nP_i \in \End_{\QQbar}(A)$. If we assume that $A$ is absolutely
    simple, we get that image of $nP_i$ must be $A$ or $0$. However since
    $(nP_1)(nP_2)=n^2P_1P_2=0$ we get that one of them must be $0$, say 
    $nP_2=0$ in ${\QQbar}(A)$. This implies that $P_2=0$,
    which contradicts our assumption that $P_2 \not \in \{0, \id\}.$
    Therefore $A$ is not absolutely simple.
\end{proof}

This lemma is used in conjunction with Ribet's result on the endomorphism algebra
of modular Abelian varieties with inner twist. Specifically let $A$ be a simple modular
Abelian variety of dimension $d$. Associated to $A$ are $d$ modular forms, Galois 
conjugate to each other. 
Let $f=\sum a_n q^n$ be
a modular form of level $N$ and weight $2$, associated to $A$.
Let $E=\QQ(\dots,a_n,\dots)$ be the field of
definition of $f$. Then we know that 
$\End_{\QQ}(A) \otimes \QQ=E$. Let $D=\End_{\QQbar}(A)\otimes \QQ$ be the algebra
of all endomorphisms of $A$. Then one easily sees that $E$ is its own commutant in $D$,
and therefore $D$ is a central simple algebra over some subfield $F$ of $E$ (see \cite{RibetEndo}).
If we assume that $A$ is absolutely simple, then $D$ must be some division algebra 
with center $E$.
Using \cite{RibetEndo} we have that $D$ must be either $E$ (which forces $E=F$) or a 
quaternion division algebra over $F$ (which forces $E$ to be a quadratic extension of $F$).
The following theorem of \cite{RibetEndo} gives us potentially good reduction everywhere.
\begin{thm}
    Suppose that $A$ has an inner twist, and that $D$ is not a matrix algebra over $F$.
    Then $A$ has potentially good reduction everywhere.
	\label{PotGoodRed}
\end{thm}

We get the following corollary.

\begin{cor}
	Let $A$ be an absolutely simple modular Abelian variety $A$ of level $N$
	with odd congruence number. Let $\delta_p=0$ for odd primes, and $\delta_2=2$.
    Assume that $p^{2+\delta_p} | N$.  
    Then $A$ has potentially good reduction everywhere. Specifically, for 
	any other prime number $q$ if $q | N$ then $q^2 |N$.
\end{cor}
\begin{proof}
	Assume that $A$ has dimension $d$, and let $f_A=\sum a_nq^n \in \CC((q))$
	be a normalized eigenform associated to $A$. Let $E=\QQ(\dots,a_i,\dots) \subset \CC.$ 
    Let $\chi$ be the quadratic character of conductor $p$.
    Since $p^{2+\delta_p} | N$, we get that $\chi \otimes f_A$ is 
    another modular eigenform
	in $S_2(\Gamma_0(N))$ (see \cite{Shim94}). Since $\chi$ is a quadratic character, $\chi$ takes 
	values in $\pm 1$, as result $\chi \otimes f_A \equiv f_A \pmod \lambda$ for any 
    $\lambda | 2$. 
	If $A$ has odd congruence number, then $\chi \otimes f_A$ must be in the same conjugacy 
    class as $f_A$. If $\chi \otimes f_A = f_A$ then $A$ has complex multiplication
	by $\chi$, and hence $A$ has potentially good reduction everywhere.
	In general $A$ might have an inner twist, and $\chi \otimes f_A = \gamma(f_A)$ for some 
    $\gamma \in \hom(E,\CC)$. Let $\Gamma \subset \hom(E,\CC)$ such that for any
    $\gamma \in \Gamma$ we can find a character $\chi_{\gamma}$ such that 
    $\chi_\gamma \otimes f_A = \gamma(f_A)$. By \cite{RibetTwists} we get that $F=E^\Gamma$,
    and as discussed above, $D=\End_{\QQbar}A\otimes \QQ$ must be a quaternion algebra.
    However, using theorem \ref{PotGoodRed} we get that $A$ has potentially
    good reduction everywhere, as desired.

	The final claim of the lemma follows by noting that if $q |N$ but $q^2 \ndiv N$, then 
	$A$ has multiplicative reduction over any field extension.
\end{proof}

We now ask what happens if $p^{2+\delta_p} |N$ and $q^{2+\delta_q} | N$ for $p$ and $q$ 
distinct primes. In this case, 
$A$ has more inner twists, and the subset $\Gamma \subset \hom(E,\CC)$ will have at
least four elements, $\gamma_1, \gamma_p, \gamma_q,$ and $\gamma_{pq}$. But that means
that $|E:F|\geq 4$, which shows $D$ must be a matrix algebra. However, lemma \ref{TateLemma}
forces $A$ not to be absolutely simple, which contradicts our assumption.
This completes the proof of the main theorem in this section.

\section{Atkin-Lehner Involution}
\label{sec22}
The goal of this section is to prove the following generalization of the main
theorem of Calegari and 
Emerton \cite{CE}.
\begin{thm}
	Let $A$ be a new simple modular Abelian variety with odd geometric congruence number.
	Assume that $A$ has no two torsion points. Then the conductor of
	$A$ is a power of a prime.
	\label{thmCE1}
\end{thm}
This theorem was proved by Calegari and Emerton in the case when $A$ is
an elliptic curve. Here we apply their techniques to higher dimensional 
modular Abelian varieties. We need to prove few lemmas first.

\begin{lemma}
	Let $f:X/k \rightarrow Y/k$ be a degree $m$ maps between curves.
	Then the composition
	\[\xymatrix@1{
		\jac(Y) \isom \jac(Y)^{\vee}  \ar[r]^-{f^*} & \jac(X)^\vee \isom 
		\jac(X) \ar[r]^-{f_*} &  \jac(Y) 
	}\]
	is just multiplication by $m$.
	\label{mult}
\end{lemma}
\begin{proof}
	It suffices to verify the above lemma for points $(z) - \infty \in \jac(Y),$
	since these points generate $\jac(Y)$. The rest of the verification
	is easy. 
\end{proof}

\begin{lemma}
	Let $\phi:J_0(N)/k \rightarrow A/k$ be a modular Abelian variety, and
	let $\pi:X_0(N)/k\rightarrow A/k$ be the composition of the Albanese
	embedding of $X_0(N)$ and $\phi$. Let $w$ be an involution on
	$X_0(N)$. Assume that the covariant action of $w$ lifts to $A$,
	which by \ref{albfunctor} implies that the Albanese induced action 
	also lifts. Assume that $\alb(w):A \rightarrow A$ is trivial.
	Then the geometric congruence exponent of $A$ is even.
	\label{lemCE}
\end{lemma}
\begin{proof}
	Let $G$ be the $2$ element group generated by the involution $w$.
	Note that $\pi \circ w$ is just the Albanese induced action of $G$
	on $A$. Therefore the conditions in the lemma are telling us that
	the induced action of $G$ on $A$ is trivial. Therefore, by lemma
	\ref{albquotient} we get that $\phi$ factors through 
	\[ \jac(X_0(N))=J_0(N) \rightarrow \jac(X_0(N)/w) \rightarrow A .\]
	Dualizing the above diagram and using the auto duality of $J_0(N)$ we
	get
	\[ \xymatrix{
	A^\vee \ar[r] \ar@{.>}[d]_\delta & \jac(X_0(N)/w)^\vee  \ar[r] \ar@{.>}[d] & J_0(N)^\vee \ar[d] \\ 
	A & \jac(X_0(N)/w) \ar[l] &  J_0(N) \ar[l] .
	} \]
	By lemma \ref{mult} the middle arrow
	is just multiplication by $2$, since degree of $X_0(N) \rightarrow X_0(N)/w$
	is $2$. Using the commutativity of the above diagram, we can see that $A^\vee[2](k) \subset \ker(\delta)$.
	Recalling that the geometric congruence number is the exponent of the kernel of $\delta$ we are done.
\end{proof}
\begin{lemma}
    Let $A/k$ be a new simple modular Abelian variety with odd geometric
    congruence number. Assume that for some Atkin-Lehner involution $w$
    we have $w_*$ is acting trivially on $A$,
    then $\alb(w)(z)=z+P$ for some $P \in A[2](k).$
    \label{twotorsionlemma}
\end{lemma}
\begin{proof}
    By lemma \ref{lemCE} we have that $\alb(w)$ is not trivial. Since
    $\alb(w)(z)=w_*(z)+P$ with $P \in A(\overline{k})$. Since
    $\alb(w^2)(z)=\alb(w)^2(z)=z+2P$ we get that $P$ is a rational
    two torsion point. Also, since $w$ is defined over $k$, we get
    that $\alb(w)$ is also defined over $k$, which implies $P \in A(k)$.
\end{proof}

Given the above lemma, we can now prove theorem \ref{thmCE1}.
\begin{proof}
    Assume that $N$ is not a power of prime. Then the group of 
    Atkin-Lehner involutions on $W$ has more than one generator, say
    $w_1$ and $w_2$ are two distinct generators. Since for any Atkin-Lehner
    involution $w \in W$ we have that $w_*(z)=\pm z$, we can find a 
    non-trivial element in $W$ such that $w_*(z)=z$. Applying lemma
    \ref{twotorsionlemma} we find $0 \neq P \in A[2](k).$
\end{proof}

\begin{remark}
	Note that in the proof of the above theorem, we can get away with slightly 
	weaker assumption than simplicity. 
\end{remark}
\begin{remark}
    Assume that $A/\QQ$ has good reduction at 2, and has odd congruence number but
    has bad reduction at at least two distinct primes. Then we can find
    an Atkin-Lehner involution such that $\alb(w)(z)=z+P$ for $P$ a two torsion
    point.
    Considering 
    $A_{\QQ_2}$ we have that $A$ has good reduction, so we can find a smooth
    model over $A_{\ZZ_2}$. Reducing this model modulo $2$ we get
    $\overline{A}$, and the map 
    \[\overline{\alb(w)}:\overline{A} \rightarrow \overline{A}.\]
    If $P$ vanishes under the reduction mod $2$ map, we get that 
    $\overline{\alb(w)}$ is trivial, which using theorem \ref{thmCE1}
    implies that the geometric congruence number is even.
    Therefore, having odd congruence number implies that $A$ has a 
    two torsion point that does not vanish
    modulo $2$.
\end{remark}

We use the rest of this section to study the action of the Atkin-Lehner involution on
$X_0(N)$ more carefully. Specifically recall the following 
\begin{lemma}
    Let $N$ be any integer, and let $r |N$ such that $(r,N/r)=1$. Then 
    $w_r:X_0(N) \rightarrow X_0(N)$ has a fixed point if and only if
    for every prime $p | r$ we have
    $-p$ is a perfect square modulo $(N/r).$
    \label{AtkinLehnerFixed}
\end{lemma}

This lemma is particularly useful because
\begin{lemma}
    Let $A$ be a modular simple Abelian variety of conductor $N$. Assume
    the Atkin-Lehner involution $w_r:X_0(N) \rightarrow X_0(N)$ has a fixed
    point. Then $(w_r)_*$ acts as $-1$ on $A$. 
    Specifically $(w_N)_*$ acts as $-1$ on $A$.
    \label{AtkinLehnerSign}
\end{lemma}
\begin{proof}
    Let $P \in X_0(N)$ be the fixed point of $w_r$. Then $\pi(P) \in A$
    is fixed under $\alb(w_r)$. However, we know that $\alb(w_r)=(w_r)_*+z$
    for some $z \in A$. Since $\alb(w_r)(\pi(P))=(w_r)_*(\pi(P))+z=\pi(P)$ 
    we get that either $\alb(w_r)$ is the identity, or $(w_r)_*$ is acting
    as $-1$, which is the desired result.

    Finally, the point $\sqrt{-N}$ is fixed by $w_N$, so $(w_N)_*$ is
    acting as $-1$.
\end{proof}

Since $(w_N)_*$ is the sign of the functional equation we get the following
\begin{cor}
    If $A$ is a simple modular Abelian variety with odd congruence number, then
    the analytic rank of $A$ is even.
\end{cor}

The following lemma helps us in dealing with even conductors
\begin{lemma}
    Let $A$ be a simple modular Abelian variety with odd congruence number
    and conductor $2M$ with $M$ odd. Then $(w_2)_*$ acts trivially on $A$.
    \label{AtkinLehnerSign2}
\end{lemma}
\begin{proof}
    We already know that $(w_{2M})_*$ will act as $-1$ on $A$, and by 
    lemma \ref{AtkinLehnerFixed} we have that $(w_M)_*$ will act as $-1$
    as well. Therefore $(w_2)_*$ must act trivially.
\end{proof}

\section{Algebraic Congruence Number}
\label{sec23}
torsion point, and an odd congruence number. By studying the cuspidal 
subgroup of $J_0(N)$, we will show that the conductor of such Abelian
varieties when $N$ is square free is the product of at most $2$ primes. 
We will then show that when $N=pq$, then $p$
and $q$ need to satisfy certain congruences.
Throughout this section, we assume that $N$ is square free.

For this section, let $N$ be the conductor of $A$, and let $\TT=\TT^\new$ be the Hecke algebra 
acting on $J_0(N)^{\new}$, 
and $S_2(\Gamma_0(N))_{\new}$. Let $\mm$ be the maximal ideal in $\TT$ generated by
$2$, $T_p-1$ for all $p|N$, and $T_l-(l+1)$ for all $l \ndiv N$.
Also let $B = \ker(\phi),$ where $\phi: J_0(N) \rightarrow A$ is an optimal quotient.

Applying the lemma \ref{AlgCongLemm}
to the maximal ideal $\mm$ we get that if $A^\dual[\mm] \neq \left\{ 0 \right\}$ and 
$B[\mm] \neq \left\{ 0 \right\}$, then the algebraic congruence number of $A$ is even.
If we show that $A^\dual[\mm] \cap B[\mm] \neq \left\{ 0 \right\}$ then we get that
the geometric congruence number of $A$ is even.

The results of this section rely on the following lemmas.
\begin{lemma}
	Let $A$ be a modular Abelian variety with a two torsion point $P$. Then
	$P \in A[\mm]$. Specifically $A^\dual[\mm] \neq 0$.
	\label{LemRib2}
\end{lemma}
\begin{proof}
	Clearly $P$ is killed by $2$, and $T_p P=-w_p P = \pm P \equiv P \pmod 2$.
	Therefore the only question is $T_l P \equiv (l+1)P$. This follows from Eichler-Shimura
	relationship $T_l(P)=(\frob_l+l/\frob_l)(P)$. Since $P$ is rational, $\frob_l$ is acting
	trivially, and the result follows.
\end{proof}

\begin{lemma}
    Let $C \subset J_0(N)$ be the cuspidal subgroup of $J_0(N)$. Then $C[2]^{2-\new}$
    is killed by $\mm$.
	\label{CuspidalLemma1}
\end{lemma}
\begin{proof}
    Clearly $C[2]$ is killed by $2$. Furthermore for any prime $l \ndiv N$ we have
    that $T_l c=(l+1)c$ for any cusp $c \in C$, so the only thing we need to check
    if $T_p c \equiv c \pmod \mm$ for $c \in C[2]^{2-\new}.$
    To do this, we use the formula of Ribet
    \[T_p+w_p=(\alpha_{N,N/p})^*\circ (\beta_{N,N/p})_*\] (see \cite{Rib91}). If $N$ is even and
    $c \in C[2]^{2-\new}$ then $(\beta_{N,N/2})_*(c)=0$ by definition. Therefore in this
    case $T_2(c)=-w_2(c)=\pm c$ which is the desired result.
    For general odd prime $p | N$, let $r |N$ such that $(r,N/r)=1$. Note that 
    $\alpha_{N,N/p}=\beta_{N,N/p}w_p$. Let $r=p^i s$ were $p \ndiv s.$ Then we have
    \begin{eqnarray*}
        (T_p+w_p)(P_r-P_1) &=& \alpha^*\beta_*(P_r-P_1) \\
        &=& \alpha^* \alpha_*(P_{p^{1-i}s}-P_p) \\
        &=& \alpha^* (P_s-P_1) \\
        &=& P_s+pP_{ps}-P_1 - pP_p \\
        &\equiv & P_s+P_{ps}-P_1-P_p \pmod 2 \\
        &\equiv & (1+w_p)(P_r-P_1) \pmod 2.
    \end{eqnarray*}
    Therefore $T_p(P_r-P_1) \equiv P_r-P_1 \pmod 2.$ Furthermore since the cuspidal
    subgroup is generated by elements $P_r-P_1$ we have that $C[2]^{2-\new}$ is killed
    by $T_p-1$.
\end{proof}

The following proposition gives us a method to show modular Abelian varieties
have even congruence numbers.
\begin{prop}
    Let $\pi:J_0(N) \rightarrow A$ be a new modular Abelian variety with odd congruence
    number. Assume that $A^\dual[\mm] \neq 0$. Let $B=\ker(\pi)$ be the orthogonal Abelian
    variety, and let $C \subset J_0(N)$ be the cuspidal subgroup of $J_0(N)$. Then
    $B \cap C[2]^{2-\new} = 0$.
    \label{CuspidalProp2}
\end{prop}
\begin{proof}
    Assume the contrary, and let $P \in B \cap C[2]^{2-\new}$ be such a point. By lemma
    \ref{CuspidalLemma1} and the fact that 
    $P \in C[2]^{2-\new}$, we have that $P$ is killed by
    $\mm$. Therefore $P \in B[\mm]$, which means $B[\mm] \neq 0.$ Now by lemma \ref{AlgCongLemm}
    we get that the characteristic of $\TT/\mm$ divides the congruence number of $A$.
    Since $\TT/\mm=\FF_2$, we get that $A$ will have even congruence number, which is 
    contrary to our assumption. Therefore $B \cap C[2]^{2-\new}=0$.
\end{proof}

We now prove the main result of this chapter.
\begin{thm}
    Let $N$ be a square free integer, and let $\pi:J_0(N) \rightarrow A$ be a new simple modular Abelian
    variety with odd congruence number. Then $N$ is either a prime number or a product of
    two prime numbers.
    \label{CongNumberPQ}
\end{thm}
\begin{proof}
    Assume that $N$ has more than two prime divisors
    and that $A$ has an odd congruence number. 
    By lemma \ref{LemRib2} we have that $A^\dual[\mm] \neq 0$,
    and therefore by proposition \ref{CuspidalProp2} we only need to create a nontrivial
    element in $B \cap C[2]^{2-\new}.$ Alternatively, we only need to create a nontrivial
    element in $C[2]^{2-\new}$ which vanishes under $\pi$. 
 
    By lemma \ref{AtkinLehnerSign} we have that $(w_N)_*$ acts as $-1$. 
    Since $w_N=\prod_{l|N}w_l$ there are odd number of primes such that $(w_l)_*$
    acts as $-1$ on $A$. Therefore we can pick distinct prime divisors
    of $N$, $p$, $q$, and $r$ such that $(w_p)_*$ acts as $-1$, while
    $(w_r)_*=(w_q)_*$. 
    Furthermore, since we know that when $N$ is even then $(w_2)_*$ acts as
    $+1$, assume that $2 \ndiv pq$.

    We now use proposition \ref{GeneralizedOgg} to construct a point of
    even order. Specifically let $s_p$, $s_q =\pm 1$. Then let the order of
    \[ z=(1-w_{qr})(1+s_p w_p)(1+s_q w_q) P_1 = (1+s_p w_p)(1+s_q w_q)(1-s_q w_r)P_1 \]
    be $m$. We have that $m$ is
    divisible by
    $\num\left( {(1+s_p p)(1+s_q q)(1-s_q r) \over 24} \right)$.
    If we choose $s_p \equiv -p \pmod 4$ and $s_q \equiv -q \pmod 4$ then
    this order is even. Therefore we have $w={m \over 2}z$ lives in
    $C[2]$. Note that if $N$ is even and $r=2$, then 
    $(\alpha_2)_*(z)=(1-s_q)(1+s_pw_p)(1+s_qw_q)P_1$. If $s_q=1$ then
    $(\alpha_2)_*(z)=0$ which means $z \in C[2]^{2-\new}$, while if 
    $s_q=-1$ then $z$ will have order $m/2$ or $m/6$. In either case,
    $(\alpha_2)_*(w)=0$ which means $w \in C[2]^{2-\new}$ 
    always.

    We now show that $\pi(z) =0$. Note that $\alb(w_{qr})(h)=a+h$ for some
    two torsion point $a$. Therefore 
    \begin{eqnarray*}
        \pi(\tau-w_{qr}(\tau))=\pi(\tau)-\alb(w_{qr})(\pi(\tau)) =a,
    \end{eqnarray*}
    for any $\tau$ in $X_0(N)$.
    Let $P=(1+s_pw_p)(1+s_qw_q)P_1=P_1\pm P_p \pm P_q \pm P_{pq}$.
    Then 
    \begin{eqnarray*}
        \pi(P-w_{qr}(P))=4a =0,
    \end{eqnarray*}
    which proves that $z \in \ker(\pi)=B.$ Therefore $(m/2) z \in B \cap C[2]^{2-\new}$
    which is the desired result.
\end{proof}

Combining this theorem with section \ref{CMcase} we get that if $A$ is an absolutely simple
modular Abelian variety with odd congruence number then the conductor of $A$ belongs to
$\left\{ 2^\alpha,p^\alpha, 2p, 4M,pq \right\}$ for
some odd prime $p$ and $q$ and positive integers $\alpha$ and $M$.
We expect that the same result is true for odd geometric congruence number. 
In general 
we conjecture the slightly stronger result
\begin{conjecture}
    If $A$ is a simple modular Abelian variety with odd geometric congruence 
    number, then the conductor of $A$ is either a power of prime,
    product of two primes, or $2^a p^b$ for $a=2$ or $3$, and $b \geq 1$.
\end{conjecture}
Using \cite{ARS} we get that when $4 \ndiv N$ then the having odd algebraic congruence number 
is the same as having odd geometric congruence number. Therefore, we can 
prove the above conjecture when $4 \ndiv N$ and $A$ is absolutely simple.

\subsection{Congruences Classes of Primes}
    Let $A$ be a simple modular Abelian variety with odd congruence number
    and conductor $N=pq$, with $p$ and $q$ odd. We want to find congruences
    that $p$ and $q$ need to satisfy. 
    By lemma \ref{AtkinLehnerSign} we know that $w_{pq}$ is acting as $-1$ on $A$.
    Therefore, we may assume without loss of generality that $w_q$ is acting
    trivially on $A$ and $w_p$ is acting as $-1$.
    Then $z=(1\pm w_p)(1- w_q)P_1$ lives in $C \cap B$, since $\pi(z)=0$.
    (The argument is the same as \ref{CongNumberPQ}. To be precise, one can
    check that $\alb(w_q)(z)=a+z$ and $\alb(w_p)(z)=b-z$ were $a$ is a two
    torsion point. Therefore
    \begin{eqnarray*}
        \pi(P_1)&=& 0, \\
        \pi(P_p)&=& b, \\
        \pi(P_q)&=& a, \\
        \pi(P_{pq})&=& a+b,
    \end{eqnarray*}
    and hence $\pi(z)=a+a=0$.)
    Now by proposition \ref{GeneralizedOgg} we have that the order of $z$ is
    $\num\left( {(p\pm 1)(q- 1) \over 24} \right).$
    Therefore if $A$ has odd congruence number, we must have that
    $p \equiv \pm 3 \pmod 8$ and $q \equiv 3 \pmod 4$.

    Similarly, if $A$ has odd congruence number with conductor $N=2p$ then 
    we know that $w_2$ acts trivially and $w_p$ acts as $-1$ on $A$.
    As before we get that $z=(1-w_2)(1\pm w_p)P_1$ lives in $C \cap B$.
    Furthermore, note that 
    \[ \alpha_*(z)=\alpha_*(P_1-P_2\pm(P_p-P_{2p}))=P_1-P_1\pm(P_p-P_p)=0, \]
    which implies that $z \in C[2]^{2-\new} \cap B$.
    The order of $z$ is $\num({p\pm 1 \over 8})$ which means
    that if $A$ has odd congruence number then $p \not \equiv \pm 1 \pmod {16}$.
    However, we also know that $w_p$ can not have any fixed points.
    This implies that $-2$ is not a quadratic residue mod $p$, which 
    in turn means that $p \equiv 5$, $7$, $13$, or $15 \pmod {16}$.
    Therefore $p \equiv 5$, $7$, or $13$.

    We collect the above in the following
    \begin{thm}
        \label{modabclassification}
        Let $A$ be a new modular Abelian variety with odd congruence number
        and semistable reduction everywhere of conductor $N$.
        Then one of the following must be true
        \begin{enumerate}
            \item $N=pq$ and $p \equiv \pm 3 \pmod 8$ and $q \equiv 3 \pmod 4$.
            \item $N=2p$ and $p \equiv 5$, $7$, or $13 \pmod {16}$.
        \end{enumerate}
    \end{thm}

\chapter{Elliptic Curves of Odd Modular Degree}
\label{chap3}
	In this chapter we specialize the results of previous chapter to the
case of elliptic curves. Doing so we will show that the rank of all 
such elliptic curves must be $0$. Furthermore we study elliptic curves
with an odd geometric congruence number, which is just the modular degree, 
When $4 \ndiv N$, then by result of Agashe, Ribet, and Stein
\cite{ARS}, having odd modular degree and odd congruence number are the
same. As such, the theorems of chapter \ref{chap2} can be stated in
terms of modular degree. 
When $4 | N$, it is possible to have odd congruence number, but
even modular degree, as such the methods of section \ref{sec23} do not
apply. However, by studying the conductor of elliptic curves with the full 
$2$-torsion structure, and studying the cuspidal subgroup more carefully
we will show that if such elliptic curves don't have conductor $4p$,
as it is conjectured by Stein and Watkins \cite{SW}, then they must satisfy
some stringent conditions. Unfortunately we are unable to rule these cases
out at this point.

Notice that for an elliptic curve $E$ we have that $E^\dual \isom E$. Therefore
we treat $E^\dual$ and $E$ as the same object.
\section{Complex Multiplication}
  If $p^2 | N$ for an odd prime $p$, then by section \ref{CMcase} we have that
  $E$ has a complex multiplication. We also showed that if $16 | N$ then 
  $E$ must have complex multiplication.
  There are only finitely many elliptic curves with complex multiplication 
  and conductor $2^mp^n$ for some prime number $p$. Here is the list of all 
  such elliptic curves that have odd modular degree: 
  $E=27A, 32A, 36A, 49A, 243B$.

  So, we will focus our attention to elliptic curves that are semistable away
  from $2$.
\section{Level $N \not \equiv 0 \pmod 4$}
 Consider an elliptic curve $E/\QQ$ such with conductor $N$. Assume that
 $E$ has odd modular degree. Then by results of \cite{ARS} we have that
 $E$ has odd congruence number.
 Therefore, applying results of previous chapter we get
 that $N$ is prime, or a prime power were $E$ has complex multiplication, or 
 product of two distinct prime numbers.
 Here we will study these cases in detail. 
 \subsection{Prime level}
  Given elliptic curve $E$ with good reduction at $2$ and $3$, one can check
  that the torsion subgroup of $E$ has size at most $5$. As result, elliptic
  curves with prime conductor have torsion of size at most $5$. Mestre and Oesterl{\`e}
  \cite{MO89} have studied elliptic curves of prime conductor, and they've showed
  that aside from elliptic curves $11A$, $17A$, $19A$, and $37B$, all such elliptic
  curves have torsion subgroup of $\ZZ/2\ZZ$ or trivial. The above cases have
  torsion structure $\ZZ/5\ZZ$, $\ZZ/4\ZZ$, $\ZZ/3\ZZ$, and $\ZZ/3\ZZ$ respectively.
  Furthermore, they show that if $E_\tors$ is $\ZZ/2\ZZ$ then $E$ is a Neumann-Setzer
  curve. In that case, the rank of the elliptic curve is $0$. We will give a proof
  of this fact, since it is fairly straightforward.
  \begin{thm}
	  Let $E$ be an elliptic curve over $\QQ$ with prime conductor $N$. Assume that
	  $E_\tors$ is nontrivial. Then $L(E,1) \neq 0$, and hence $E(\QQ)$ has rank
	  $0$.
  \end{thm}
  \begin{proof}
      Recall that 
      \[L(E,1)=2\pi i \int_0^{i \infty} f_E(z) dz \equiv \pi(P_N) \pmod \Lambda_E,\]
      were $\CC/\Lambda_E \isom E(\CC)$.
      Therefore if $L(E,1)=0$ then $\pi(P_N)=0$, or alternatively $\pi_*(P_1-P_N)=0$.
      By \cite{Maz79} we know that $J_0(N)_\tors$ is generated by the cusp $P_1-P_N$,
      and for any Abelian quotient of $J_0(N) \rightarrow A$, we have $A_\tors$ is generated
      by the image of $\pi_*(P_1-P_N)$. Since we are assuming that $E$ has nontrivial torsion
      structure, we must have that $\pi_*(P_1-P_N)$ is nontrivial, which implies $L(E,1) \neq 0$.
  \end{proof}

  When $E$ has trivial torsion structure, then we are currently unable to prove any positive
  result for $E$. In the next chapter we recall 
  an argument of Dummigan to justify
  Stein and Watkins conjecture of the ranks of elliptic curves.
 \subsection{Level $N=2p$}
  When $N$ is a product of two distinct primes, computer calculation shows
  us that the even conductor and odd conductors behave differently. In this
  subsection, we will study the even case that is $N=2p$ with $p$ an odd
  prime. Specifically, we want to show that $L(E,1) \neq 0$. 
  One can prove this by studying the cusps in $J_0(N)$, however in this case
  it seems easier to prove this using analytic tools. 

  Specifically let $f_E(q)=\sum a_nq^n$ be the modular form attached to the 
  elliptic curve $E$, and let $\Omega_E$ be the real period of $E$.
  Note that $L(f_E, 1) \in \RR$ since the fourier coefficients of $f_E$ are
  rational integers. Therefore the order of $\pi(P_{2p})$ is the order of
  $L(f_E,1) \in \RR/{\Omega_E \ZZ}$.
  We know that $L(f_E,s)$ has an Euler expansion
  \[ L(f_E,s)= \prod_p L_p(f_E,s),\]
  and we have $L_2(f_E,s)={1 \over 1-a_22^{-s}}.$
  Similarly 
  \begin{eqnarray*}
      \pi(P_{p})&=& 2\pi i\int_{1\over 2}^{i\infty} f_E(z)dz \\
      &=& 2\pi i \int_0^{i\infty} f_E(z+1/2)dz  \\
      &=& 2\pi i \int_0^{i\infty} \sum (-1)^{n}a_n q^ndz
  \end{eqnarray*}
  which implies that $\pi(P_p)$ can be written as $L(g,1)$ were $L(g,s)$
  has an Euler product expansion
  \begin{eqnarray*}
      L(g,s)&=&  ({-1+{a_2 \over 2^s}+{a_4\over 4^s}+\dots})\prod_{p>2} L_p(f_E,s) \\
      &=& -{1-a_2 2^{1-s} \over 1-a_2 2^{-s}} \prod_{p>2}L_p(f_E,s)
  \end{eqnarray*}
  Therefore $L(g,1)=L(f_E,1)(a_2-1),$
  and more appropriately for us 
  \[ \pi(P_p)\equiv (a_2-1)\pi(P_{2p})\pmod {\Omega_E \ZZ}.\]
  
  We know that if $E$ has an odd congruence number, then $w_2$ is acting
  trivially, which implies that $a_2=-1$. Therefore
  \[ \pi(P_p) \equiv -2\pi(P_{2p})  \pmod {\Omega_E \ZZ}.\]
  However, we also know that $P_{2p}=w_2(P_{p})$, and $\pi(w_2(P_{p}))=\pi(P_{p})+\alpha$
  were $\alpha$ is a two torsion point in $E$. Furthermore since both
  $\pi(P_p)$ and $\pi(P_{2p})$ are both equivalent to a real number, then
  we get that $\alpha$ is equivalent to a real number as well, which 
  implies $\alpha \equiv {\Omega_E \over 2} \pmod {\Omega_E \ZZ}.$
  Therefore
  \begin{eqnarray*}
      \pi(P_p) &\equiv& \pi(P_{2p})+{\Omega_E \over 2} \pmod {\Omega_E \ZZ}, \\
      & \equiv & -2\pi(P_{2p}) \\
      \Rightarrow -3\pi(P_{2p}) &\equiv& {\Omega_E \over 2} \pmod {\Omega_E \ZZ}, \\
      \Rightarrow \pi(P_{2p}) &\equiv& \Omega_E( {k \over 3}-{1 \over 6}) \pmod {\Omega_E \ZZ}
  \end{eqnarray*}
  for some integer $k$. Therefore $\pi(P_{2p}) \neq 0$ and hence 
  $L(f_E,1) \neq 0$. Furthermore we know that $\pi(P_{2p})$ will either
  be a $6$ torsion point (for $k \equiv 0$ or $1 \pmod 3$), or a two
  torsion point (for $k \equiv 2 \pmod 3$).

      Note that in either case, we have an elliptic curve with conductor
      $2p$ having a rational two torsion points. Such elliptic curves have
      been studied by Ivorra \cite{Ivorra}, and one can use his techniques
      to put stringent conditions on what values $p$ can be.
      In particular he shows that if $p \geq 29$ then 
      there is an integer $k \geq 4$ such that
      one of $p+2^k$, $p-2^k$, or $2^k-p$ is a perfect square.
      However, we already know from theorem \ref{modabclassification} that 
      $p \equiv 5$, $7$, or $13 \pmod{16}$. Putting these
      two together we get that $p \equiv 7 \pmod {16}$, and 
      $p=2^k-m^2$.
      In fact, in this case Ivorra's result tell us that
      $7 \leq k < f(p)$ were
      \begin{eqnarray*}
      f(p)= 
      \begin{cases}
          18+2\log_2 n & \mbox{ if $n<2^{96}$,} \\
          435+10\log_2 n & \mbox{ if $n>\geq 2^{96}$}
      \end{cases},
      \end{eqnarray*}
      and our elliptic curve is isogeneous to
      \[ y^2+xy=x^3+{m-1 \over 4}x^2+2^{k-6}x.\]
      Furthermore, quick search through the Cremona database, shows us
      that the only elliptic curves with odd modular degree and conductor
      $2p$ with $p \leq 29$ are $E=14A$ and $E=46A$, and both of these
      are of the form above.

 \subsection{Level $N=pq$}
  In this subsection, we will study elliptic curves of odd modular degree 
  and conductor $pq$ were $p$ and $q$ are both odd.
  By theorem \ref{modabclassification} we know that $p \equiv \pm 3 \pmod 8$
  and $q \equiv 3 \pmod 4$. We will show that with few exceptions, 
  $p \equiv 3 \pmod 8$ and $q \equiv 3 \pmod 8$. Given this choise,
  $w_q$ is acting trivially on $E$.
  Furthremore, all such elliptic curves have rank $0$ over $\QQ$.

  We will first show that $E[2] = \ZZ/2 \times \ZZ/2$. 
  From section \ref{sec22} we have that 
  $\alb(w_p)$ is translation by a two torsion point $P \in E[2]$.
  Considering this map over $\FF_2$, we get that this point $P \in E[2]$
  will not vanish under the reduction mod $2$ map.
  \begin{lemma}
    Let $E$ be a rational elliptic curve with conductor $pq$ with an
    odd congruence number, as discussed earlier in this section.
    If $(p,q) \not \equiv (3,3) \pmod 8$, then 
    $E[2](\QQ)=(\ZZ/2)^2$
	\label{Technical1}
  \end{lemma}
  \begin{proof}
    Let $P \in E[2](\QQ)$ be a rational two torsion point that does not vanish
    mod $2$. Assume that $E[2](\QQ)=\ZZ/2$.
	Say $E$ has a torsion point $Q$ of order $2m+1$. Then $\#E(\FF_2)\geq 2(2m+1) \geq 6$
	which contradicts the Hasse-Weil bound. Similarly, assume
	that $E$ has a rational $8$ torsion point $Q \in E(\QQ)$. 
    Then $4Q$ has order $2$. 
	Since we are assuming $E[2]=\ZZ/2$
	we get that $4Q=P$. Applying the reduction map $\pi$ we get that 
    $4\pi(Q)=\pi(P) \neq 0$.
	Therefore $\pi(Q)$ has order $8$ in $E(\FF_2)$. This contradicts the Hasse-Weil bound.
	Therefore $E_\tors=\ZZ/4$ or $\ZZ/2$.

	Now assume that $E_\tors=\ZZ/4$. As discussed earlier $\alb(w_p)$ is
    translation by $P$. Let $\alb(w_q):z \goto Q-z$.
    If $Q$ has order $4$ then $2Q=P$. In that case
    \[ \pi_*(P_{pq}+P_q-2P_1)=(Q+P)+(Q)-2(0)=0.\]
    However, by \ref{ligozatprop} we get that 
    $T=P_{pq}+P_q-2P_1$ has an even order. To see this note that
    \[\Lambda^{-1}T={24 \over (p^2-1)(q^2-1)} \begin{pmatrix}
        1-p-2pq \\
        1-p+2q \\
        2p-q+pq \\
        -2-q+pq
    \end{pmatrix}.\]
    Since $p \equiv \pm 3 \pmod 8$ and $q \equiv 3 \pmod 4$, we get $v_2(p^2-1)=3$
    and therefore this has even order if some element
    in the above vector has $2$-valuation less than the $2$-valuation of
    $q^2-1$. Since $q$ is odd, $v_2(q^2-1)\geq 3$, while for all possible
    congruences we that $v_2(2p-q+pq) \leq 2$.
    Therefore $C[2] \cap B$ is not empty, which implies $A$ must have
    even congruence number.

	Now assume that $E_{\tors}=\ZZ/2$. Then, either $P_1-P_{pq}$ goes to the
    origin or $P_1-P_q$ goes to the origin.
    We know that the order of $P_1-P_q=(p^2-1)(q-1)/24$ which is even,
    hence for $E$ to have odd congruence number we get $P_1-P_{pq}$ goes
    to the origin.
    In that case, $P_1-P_{pq}$ has an odd order if and only if 
    $p \equiv q \equiv 3 \pmod 8$.
  \end{proof}

  Now we prove a lemma that we need to show elliptic curves of odd
  congruence number in our situation have rank $0$.
  \begin{lemma}
      Let $E$ be a semistable elliptic curve with conductor $pq$
      and $p \equiv q \equiv 3 \pmod 8$, having
      a two torsion point, then $E[2](\QQ)=(\ZZ/2)^2.$ 
      \label{technical4}
  \end{lemma}
  \begin{proof}
    Notice that elliptic curves with rational $2$-torsion points
    and good reduction at $2$ have a model
    \[ E: y^2+xy=x^3+a_2x^2+a_4x.\]
    Recall that $b_2=4a_2+1$, $b_4=2a_4$, $b_6=0$, and $b_8=-a_4^2$.
    The discriminant of $E$ is 
    \[ \Delta=a_4^2 ( (4a_2+1)^2-64a_4),\]
    and we have
    \begin{equation}
        x([2]Q)={x^4-b_4x^2-b_6x-b_8 \over 4x^3+b_2x^2+2b_4+b_6},
        \label{dblform}
    \end{equation}
    for point $Q=(x,y).$ 
    Therefore $E[2]=(\ZZ/2)^2$ if and only if the cubic $4x^3+b_2x^2+2b_4+b_6$
    will split completely, which means that $a_4^2((4a_2+1)^2-64a_4)$ will be
    a perfect square.
    Assume the contrary, that is $(4a_2+1)^2-64a_4$ is not a perfect
    square. Also, since we are assuming that $E$ has potentially
    good reduction everywhere, we have that $4a_2+1$ and $a_4$ are 
    coprime to each other.
    There are few cases that we need to consider $a_4=\pm 1$, $\pm p^\alpha$
    or $\pm p^\alpha q^\beta$. 
    \begin{enumerate}
        \item If $a_4=-1$ then $(4a_2+1)^2+64=\pm p^\alpha q^\beta$.
            Therefore $p^\alpha q^\beta$ is a sum of two squares, and since they
            are both equivalent to $3 \pmod 8$ we get that $\alpha$ and
            $\beta$ must be even. Therefore $\Delta$ is a perfect square.
        \item
            If $a_4=1$ then we get 
            \[ (4a_2+1)^2-64=(4a_2+9)(4a_2-7)=\pm p^\alpha q^\beta.\]
            If $4a_2+1 \pm 8=\pm 1$ then we get $p^\alpha q^\beta=17$
            or $15$. Neither of these cases $p \equiv q \equiv 3 \pmod 8$.
            Therefore we must have $4a_2+9=p^\alpha$ and $4a_2-7=q^\beta$.
            However this forces $p^\alpha \equiv q^\beta \equiv 1 \pmod 4$,
            which implies $\alpha$ and $\beta$ are even, and hence
            $\Delta$ is a perfect square.
        \item If $a_4=\pm p^\gamma$ then we have
            $(4a_2+1)^2-64(\pm p^\gamma)=\pm q^\beta$.
            Considering the above equality modulo $4$ we get that
            $\beta$ is even.
            Therefore $\Delta=a_4^2( (4a_2+1)^2-64a_4)=p^{2\gamma} q^\beta$
            is a perfect square.
        \item If $a_4=\pm p^\gamma q^\tau$ then we get 
            $(4a_2+1)^2-64(\pm p^\gamma q^\tau)=\pm 1$. Again, considering
            the above modulo $4$ we get that the right hand side equals
            $+1$. 
            Therefore $\Delta=a_4^2( (4a_2+1)^2-64a_4)=p^{2\gamma} q^{2\tau}$
            is a perfect square.
    \end{enumerate}
    So, in all cases we get that $\Delta$ is a perfect square, and hence
    $E[2](\QQ)=(\ZZ/2)^2.$
  \end{proof}
  \begin{cor}
      Let $E$ be an elliptic curve of conductor $pq$ with odd congruence
      number. Then $L(E,1) \neq 0$, and hence $E$ has rank $0$.
  \end{cor}
  \begin{proof}
      We've already seen that in this case $E[2](\QQ)=(\ZZ/2)^2.$ 
      Also, we know that $E_{\tors}$ is generated by the image of cusps.
      We saw that $\pi(P_p)=P$ is a two torsion point. Therefore
      $\pi(P_{pq})$ can not be $0$, otherwise $\pi(P_q)=P$ and hence
      $E_{\tors}=\ZZ/2$. Therefore $L(E,1) \neq 0$, and $E$ has
      rank $0$.
  \end{proof}

  Now we show that when $E_{\tors}$ is $\ZZ/2 \times \ZZ/4$, then 
  $E$ has conductor $15$ or $21$. First we prove a general result
  about semistable elliptic curves with $E_{\tors}=\ZZ/2 \times \ZZ/4$ and 
  good reduction at $2$.
  Specifically
  \begin{lemma}
      Let $E$ be a semistable elliptic curve with good reduction at $2$.
      $E_{\tors}=\ZZ/2 \times \ZZ/4$, and let $Q \in E(\QQ)$ is a point
      of order $4$. Then $Q$ has order $4$ in $E(\FF_2)$.
      \label{technical3}
  \end{lemma}
  \begin{proof}
    We use the same notation as lemma \ref{technical4}, so
    \[ E: y^2+xy=x^3+a_2x^2+a_4x.\]

    Since $E[2] =\ZZ/2 \times \ZZ/2$ we get that
    $(4a_2+1)^2-64a_4$ and $\Delta$ are perfect squares.
    Let $(4a_2+1)^2-64a_4=m^2$.
    The $x$ coordinate of the two torsion points are $0$, $4\alpha$, and 
    ${\beta \over 4},$ were we've chosen $m$ so that $\alpha$ and $\beta$ are 
    both (odd) integers. Furthermore, they better be coprime to each other for
    $E$ to be semistable. Note that ${\beta \over 4}$ maps to the origin
    under the reduction mod $2$ map.
    We have 
    \begin{eqnarray*}
        b_2&=& 16 \alpha+\beta, \\
        b_4&=& 2\alpha \beta, \\
        b_6&=& 0, \\
        b_8&=& \alpha^2 \beta^2,\\
        \Delta&=& \alpha^2 \beta^2 (16\alpha-\beta)^2.
    \end{eqnarray*}
    Now let $Q$ be a point of order $4$. Then $x([2]Q)=0$, $4\alpha$, or ${\beta \over 4}.$
    If $x([2]Q)={\beta \over 4},$ then we get
    \begin{eqnarray*}
        {\beta \over 4} &=& {x^4-b_4x^2-b_8 \over 4x^3+b_2x^2+2b_4x} \\
        &=& {x^4-2\alpha \beta x^2 - \alpha^2 \beta^2 \over 4x^3+(16 \alpha+\beta)x^2+4\alpha^2 \beta^2}, \\
        \Rightarrow 0 &=&  x^4+\beta x^3+(2\alpha \beta+{\beta^2 \over 4}x^2+\alpha \beta^2 x + \alpha^2 \beta ^2 \\
        &=& (x^2+{\beta \over 2} x + \alpha \beta \\
        &=& (x+{\beta \over 4})^2-{\beta^2 \over 16}+\alpha \beta.
    \end{eqnarray*}
    Therefore we must have $-{\beta^2 \over 16}+\alpha \beta={\beta^2 \over 16}(-\beta+16\alpha)$
    is a perfect square. Since $\alpha$ and $\beta$ are coprime to each other, we get
    that $\beta$ and $-\beta+16 \alpha$ are both perfect square. This can't happen,
    since only one of those two can be congruent to $1 \pmod 4$.
    Therefore $x([2]Q)=0$ or $4\alpha$. In either case, under the reduction
    mod $2$ map we get that $Q$ has order $4$, which is the desired result.
  \end{proof}

  \begin{lemma}
	Let $E$ be an elliptic curve of conductor $pq$ with $E_{\tors}=\ZZ/2 \times \ZZ/4$.
	Then $pq=15$ or $21$.of conductor $N=pq$ and
	\label{Technical2}
  \end{lemma}
  \begin{proof}
    Using the same notation as lemma \ref{technical3},
    let $0$, $4\alpha$ and $\beta\over 4$ be the $x$-coordinate of
    $2$ torsion points of $E$.
    Let $Q$ be a point in $E_{\tors}$ of order $4$.
	If $x([2]Q)=4\alpha$, we can do a change of coordinate to find another model
	with $x([2]Q')=0$. Therefore without loss of generality assume that $x([2]Q)=0$.
    Using the double point formula \ref{dblform} we get that
    $(x^2-a_4)=0$, therefore $a_4$ must be a perfect square.
	Since $\Delta=a_4^2\left(  (4a_2+1)^2-64a_4 \right)=p^{2r} q^{2s},$ we
	get that $a_4=1$, $p^r$, $q^s$, or $p^rq^s$.
	\begin{enumerate}
	\item[$a_4=1$]
		In this case we get $(4a_2+1)^2-64=p^{2r}q^{2s}=m^2.$ Therefore 
		\[ (4a_2+1+m)(4a_2+1-m)=64. \]
		Choosing $m$ with the correct sign, we get $4a_2+1-m=\pm 2$ and 
		$4a_2+1+m=\pm 32$ which gives us $m=\pm 15$. This forces the conductor
		to be $15$.
	\item [$a_4=p^r$ ]
		First note that $r$ is even. In this case we get $(4a_2+1)^2-64p^{r}=q^{2s}$
		Factoring the right side we get
		\[ (4a_2+1+8p^{r/2})(4a_2+1-8p^{r/2})=q^{2s},\]
		which implies $4a_2+1-\pm 8p^{r/2}=1$ and $4a_2+1\pm 8p^{r/2}=q^{2s}$.
		Subtracting we get $q^{2s}-1=\pm 16 p^{r /2}.$ Since $q^{2s}>1$ we have
		the positive sign.
		Factoring the right side we get $q^s \pm 1= 8$ or $2$, and 
		$q^s-\pm 1= 2p^{r /2}$ or $8p^{r/2}$. Considering all cases
		we get $q^s=7$ or $9$ with $p^r=9$ or $25$. This forces the conductor
		to be either $15$ or $21$.
	\item [$a_4=p^rq^s$]
		Then we have $(4a_2+1)^2-64a_4=1$. However, there are no positive
		perfect squares that differ by $1$, therefore this case can't happen.
	\end{enumerate}
	This completes our proof.
  \end{proof}

  \begin{remark}
	The previous lemma is long, and tedious. It certainly feels like certain
	parts can be shortened, although I haven't figured out how. Note that
	it is really easy to show that $3$ must divide the conductor by the Hasse-Weil
	bound. Unfortunately I don't see how that can simplify the argument.
  \end{remark}

  \begin{thm}
	Assume $E$ is an elliptic curve with odd modular degree. Furthermore, assume
	that conductor of $E$ is $pq$ with $pq > 21$. Then $p, q \equiv 3 \pmod 8$.
  \end{thm}
  \begin{proof}
	Since $E_{tors}=\ZZ/2 \times \ZZ/2$ we get that the cusps
	\begin{eqnarray*}
		D^{+-}&=& P_1+P_p-P_q-P_{pq}, \\
		D^{-+}&=& P_1-P_p+P_q-P_{pq}, \\
		D^{--}&=& P_1-P_p-P_q+P_{pq}.
	\end{eqnarray*}
	One can check that the order $D^{ab}$ is $\num\left( {(p+a)(q+b) \over 24} \right)$.
	Furthermore since $P_i$'s all map to a two torsion point and since
	$P_{pq}$ maps to the sum of the images of $P_p$ and $P_q$, we get that
	each $D^{ab}$ vanishes in $E$. Therefore, by theorem \cite{ARS} we get
	that $D^{ab}$ all have odd orders in $J_0(N)$.
	
	Considering $D^{--}$ we get that $16 \ndiv (p-1)(q-1)$. Therefore either $p$ or
	$q$ is congruent to $3 \pmod 4$. Assume without loss of generality that
	$p \equiv 3 \pmod 4$.
	Now considering $D^{+-}$ we get that $16 \ndiv (p+1)(q-1)$. This gives us
	that $p \equiv 3 \pmod 8$ and $q \equiv 3 \pmod 4$.
	Finally considering $D^{-+}$ we get that $16 \ndiv (p-1)(q+1)$, which implies
	$q \equiv 3 \pmod 8$.
  \end{proof}

\section{Level $N \equiv 0 \pmod 4$}
    When $E$ has additive reduction at $2$, we have to work quite a bit harder
    to get the results that we expect. In this case, as we mentioned before,
    the modular degree can be odd, even if the congruence number is even.
    Furthermore, we can not apply the methods of \ref{CMcase} unless
    $16 | N$, and as such there are many more cases that we need to 
    consider. However we still know that if an odd prime $p |N$, then by 
    thereom \ref{thmCE1} we get that $E$ has a rational two torsion point.
    Furthermore, if two distinct odd primes $p$ and $q$ divide $N$ then
    we get that $E[2](\QQ) \isom (\ZZ/2)^2,$ and finally no more than
    two odd primes can divide the conductor, since $\#E[2](\QQ) | 4$.

    In this section, we study these cases by considering elliptic curves
    with the desired torsion structure and conductor. Throughout this
    section, we will assume that $4 | N$.

    \subsection{Case $N=2^r pq$}
        We will first study what happens when $E$ has bad reduction at
        two distinct primes $p$ and $q$, and having odd congruence number.
        By results of section \ref{CMcase} we get that $v_2(N)<4$, otherwise
        $E$ must be a power of $2$.
        By theorem \ref{thmCE1} we have that
        $E[2](\QQ)=(\ZZ/2)^2$, and thus $E$ has a minimal model
        \[ E: y^2=x(x+a)(x+b),\]
        with discriminant $\Delta=2^4(ab(a-b))^2$.
        If needed, we can translate $x$ so that $b$ is even, and
        $a \neq \pm 1$.
        Since $E$ has multiplicative reduction at $p$ and $q$, we get
        that $\gcd(a,b,(a-b))=2^u$ for some $u$.
        There are few cases that we need to consider.
        Note that if $2|\gcd(a,b)$, then we may substitute
        $x'=2x$ and $y'=4y$ to get
        \[ E: 2(y')^2=x'(x'+a/2)(x'+b/2),\]
        which is just a twist of 
        \[ E': y^2=x(x+a/2)(x+b/2).\]
        Since conductor of $E'$ divides $E$,
        we get that $E$ will have even congruence number.
        Therefore, we have two cases that we need to consider: either
        $a=\pm p^\alpha q^\beta$ with $\alpha$ and $\beta>0$ and 
        $b=\pm 2^\gamma$ or 
        $a=\pm p^\alpha$ and $b=\pm 2^\gamma q^\beta$.
        We apply Tate's algorithm to calculate the conductor of such
        elliptic curves in each case.
        \begin{itemize}
            \item[Case 1] 
                Assume $a=s_1 p^\alpha q^\beta$ and $b=s_2 2^\gamma$,
                were $s_i^2=1$.
                Furthermore, assume that $\alpha, \beta>0$.
                Then using the notation of \cite{AECII} we 
                $a_2=(s_1 p^\alpha q^\beta + s_2 2^\gamma)$, 
                $a_4=s_1s_2 p^\alpha q^\beta w^\gamma$, and 
                $a_1=a_3=a_6=0$.
                Furthermore, 
                $\Delta=2^{4+2\gamma}p^{2\alpha}q^{2\beta}(p^\alpha q^\beta-s_1s_22^\gamma)^2.$
                Applying Tate's algorithm as presented in \cite{AECII}, page
                365, we see that when $\gamma=0$ or $1$
                then we end up in step 4, and $E$ will have Kodaira type
                $III$, and $v_2(N)=v_2(\Delta)-1 \geq 5$.
                Therefore $\gamma>1$. In that case, if $s_1=s_2$ then
                we end up in step 7, and $E$ will have Kodaira type
                $I_n^*$, were $n=2(\gamma-2)$. In this case 
                \[v(N)=v(\Delta)-n-4=4+2\gamma-2(\gamma-2)-4=4.\]
                Again, this contradicts the odd congruence number assumption.
                If $s_1=-s_2$ then if $\gamma=2$ we have type $I_n^*$
                again for $n=2(\gamma-2)$ which gives us $v(N)=4$.
                If $\gamma>3$ we get that $E$ has type $III^*$ and 
                $v(N)=3$. This does not contradict our assumption
                for odd congruence directly. However in this case,
                since $\Delta=2^* p^* q^*$ we get that
                $p^\alpha q^\beta-s_1s_2 8=\pm 1$. Therefore
                we get $p^\alpha q^\beta=9$ or $7$, which can't happen
                by our initial assumption.
                Therefore $\gamma\geq 4$. However, in this case our
                elliptic curve was not minimal, and changing coordinates
                we get that $E$ in fact has semistable reduction at $2$.

            \item[Case 2] 
                Assume $a=s_1 p^\alpha$ and $b=s_2 2^\gamma q^\beta$,
                were $s_i^2=1$.
                We get that
                $\Delta=2^{4+2\gamma}p^{2\alpha}q^{2\beta}(p^\alpha -s_1s_22^\gamma)^2.$
                Again, applying Tate's algorithm we get that if
                $\gamma=0,1$ then we are in type $III$ and $v_2(N)=5$,
                which is not possible.
                So $\gamma \geq 2$. If $s_1 p^\alpha \equiv 3 \pmod 4$
                then the Kodaira symbol for $E$ is $I_n^*$ were
                $n=2(\gamma-2)$ and $v(N)=4$, which contradicts
                the odd congruence number assumption.
                If $s_1 p^\alpha \equiv 1 \pmod 4$ and
                $\gamma>3$, then our elliptic curve was not minimal,
                and with a change of coordinate we get that $E$ in
                fact had semistable reduction.
                However, when 
                $\gamma=2$ we end up in step 7, with Kodaira symbol
                $I_1^*$ and 
                $v(N)=v(\Delta)-4-1=3$. Similarly if $\gamma=3$ we end
                up in step 9, with Kodaira symbol $III^*$ and 
                $v(N)=3$. Unfortunately, we can't rule out either of
                these cases using our techniques, and all we can say
                at this point is that
                $|p^\alpha-q^\beta|=4$ or $8$.

                However, one can check that $|p^\alpha-q^\beta|=4$
                then $E$ will have potentially good reduction.
                That means after making a base change to $K$,
                $E$ will attain good reduction. However, the
                degree of the map $X_0(N) \rightarrow E$ will not
                change, and now by reducing this map modulo a prime
                above $2$, we get that the Albanese induced action
                of one of Atkin-Lehner involutions must be trivial.
                Therefore the degree of the map must be even.
                Therefore, the only case that we can't deal with
                is $|p^\alpha-q^\beta|=8$.
        \end{itemize}

        We summarize the results of the above calculation in the following
        \begin{prop}
            If $E$ is an elliptic curve with odd congruence number and
            conductor $2^r pq$, then $f=3$ and $E$ has the form
            \[ E: y^2=x(x-s p^\alpha)(x-s q^\beta),\]
            for some $s^2=1$ and $\alpha$, $\beta$ positive integers,
            and $|p^\alpha-q^\beta|=8$.
        \end{prop}
        We remark that we present $E$ in a slightly different form than
        what we used in our calculation, since it makes the proposition
        easier to read.
        We don't expect such elliptic curves to actually exist, and in
        fact Stein and Watkins's computations and conjecture suggests that
        there are no such elliptic curves.
        One can probably come up with a complete proof of this result
        by studying the cuspidal subgroup of $J_0(N)$ when $N=8pq$,
        and generalizing the results in section \ref{sec23}.

    \subsection{Case $N=2^rp$}
        In this case, we have that $E$ has a rational two torsion point,
        and therefore we can use Ivorra's classification of such 
        elliptic curves. When $r=2$, then we get that $p=m^2+4$ for
        some integer $m$ and
        \[ E: y^2=x(x^2+mx-1).\]
        We conjecture that infinitely many such elliptic curves have
        odd modular degree. 

        When $r=3$, again Ivorra has a classificaion, and we get 
        that for $p>31$ we are in 
        one of the following cases:
        \begin{enumerate}
            \item The integer $p-16$ is a square and $E$ is isogeneous to
                \[ y^2=x^3+\sqrt{p-16} x^2-4x,\]
            \item the integer $p-32$ is a square and $E$ is isogeneous to
                \[ y^2=x^3+\sqrt{p-32}x^2-8x,\]
            \item the integer $p+32$ is a square and $E$ is isogeneous to
                \[ y^2=x^3+\sqrt{p+32}x^2+8,\]
        \end{enumerate}
        By searching Cremona's database for elliptic curves of
        conductore less than $30000$ we only find elliptic curve $24A$
        of odd modular degree. We conjecture that $24A$ is the only 
        elliptic curve of odd modular degree with $8$ dividing its
        conductor.

\clearpage

\bibliographystyle{plain}
\bibliography{refs}

\end{document}